\documentclass[10pt,reqno]{amsart}

\usepackage{amsmath}
\usepackage{amsfonts}
\usepackage{epsfig}
\usepackage{float}

\numberwithin{equation}{section}
\newtheorem{definition}{Definition}[section]
\newtheorem{theorem}[definition]{Theorem}
\newtheorem{lemma}[definition]{Lemma}
\newtheorem{proposition}[definition]{Proposition}

\usepackage{xcolor}

\def\vec#1{\mathchoice{\mbox{\boldmath$\displaystyle#1$}}
{\mbox{\boldmath$\textstyle#1$}}
{\mbox{\boldmath$\scriptstyle#1$}}
{\mbox{\boldmath$\scriptscriptstyle#1$}}}

\newcommand\dd{{\mathrm d}}
\newcommand\ism{\cong}

\newcommand\nix{\,\cdot\,}
\newcommand{\toboss}{\uparrow}
\newcommand{\fromboss}{\downarrow}

\newcommand{\edgel}{\cbc{000,001,010,111}}
\newcommand{\vertexl}{\cbc{000,001,010,110,111}}
\newcommand{\labelset}{\cbc{000,001,010,110,111}}

\newcommand\atom{\delta}
\newcommand\thet{\vartheta}
\newcommand\eps{\varepsilon}

\newcommand\vv{\vec v}
\newcommand\vF{\vec F}
\newcommand\vG{\vec G}
\newcommand\vH{\vec H}
\newcommand\T{\vec T}
\newcommand\cB{\mathcal{B}}
\newcommand\cC{\mathcal{C}}
\newcommand\cE{\mathcal{E}}
\newcommand\cF{\mathcal{F}}
\newcommand\cT{\mathcal{T}}
\newcommand\cI{\mathcal{I}}
\newcommand\cJ{\mathcal{J}}
\newcommand\cL{\mathcal{L}}
\newcommand\cP{\mathcal{P}}
\newcommand\cX{\mathcal{X}}
\newcommand\cV{\mathcal{V}}
\newcommand\fF{\mathcal{F}}
\newcommand\fG{\mathcal{G}}
\newcommand\RR{\mathbb{R}}

\newcommand\core{\cC}

\DeclareMathOperator{\pr}{\mathbb P}

\newcommand\Erw{\mathbb{E}}

\newcommand{\vecone}{\vec{1}}

\newcommand{\Po}{{\rm Po}}
\newcommand{\Bin}{{\rm Bin}}
\newcommand{\Be}{{\rm Be}}

\newcommand{\bink}[2] {{{#1}\choose {#2}}}

\newcommand\br[1]{\left(#1\right)}
\newcommand\bc[1]{\left({#1}\right)}
\newcommand\cbc[1]{\left\{{#1}\right\}}
\newcommand\brk[1]{\left\lbrack{#1}\right\rbrack}
\newcommand\abs[1]{\left|{#1}\right|}

\newcommand{\whp}{w.h.p.}

\newcommand{\Erdos}{Erd\H{o}s}
\newcommand{\Renyi}{R\'enyi}

\newcommand{\Bollobas}{Bollob\'as}

\newcommand\Lem{Lemma}
\newcommand\Prop{Proposition}
\newcommand\Thm{Theorem}
\newcommand\Chap{Chapter}

\newcommand\RSA{Random Structures \& Algorithms}
\newcommand\JCTB{Journal of Combinatorial Theory, Series~B}

\title{The core in random hypergraphs and local weak convergence}
\author{Kathrin Skubch$^*$}
\thanks{$^*$Goethe University, Mathematics Institute, 10 Robert Mayer St, Frankfurt 60325, Germany, skubch@math.uni-frankfurt.de}

\begin{document}
\maketitle
\begin{abstract}
\noindent The degree of a vertex in a hypergraph is defined as the number of edges incident to it.
In this paper we study the $k$-core, defined as the maximal induced subhypergraph of minimum degree $k$, of the random $r$-uniform hypergraph $\vH_r(n,p)$ for $r\geq 3$. 
We consider the case $k\geq 2$ and $p=d/n^{r-1}$ for which
asymptotically
 every vertex has fixed average degree $d/(r-1)!>0$. We derive a multi-type branching process
that describes the local structure of the $k$-core together with the mantle, i.e. the vertices outside the $k$-core. This paper
provides a generalization of 
prior work on the $k$-core
problem in random graphs [A. Coja-Oghlan, O. Cooley, M. Kang,
K. Skubch: 
How does the core sit inside the mantle?, \RSA~{\bf 51} (2017) 459--482].
\bigskip
\noindent



\end{abstract}
\section{Introduction}
\noindent{\em 
Let $r\geq 3$ and let $d>0$ be a fixed constant. We denote by $\vec H=\vec H_r(n,d/n^{r-1})$ the random $r$-uniform hypergraph on the vertex set $[n]=\cbc{1,\ldots,n}$
	in which any subset $a$ of $[n]$ of cardinality $r$ forms an edge with probability $p=d/n^{r-1}$ independently.
We say that the random hypergraph $\vec H$ enjoys a property with high probability (`\whp') if the probability of this event tends to $1$ as $n$ tends to infinity.}

\subsection{Background and motivation}\label{sec_intro}
Since the notion of random graphs was first introduced by \Erdos~and \Renyi, many striking results in this area were achieved. One important 
theme is the emergence of the giant component in this model \cite{BB,JLR}. In their seminal paper from 1960 \cite{ER} \Erdos~and
\Renyi~proved that the structure of the sparse random graph $\vG(n,p)$ with $p=d/n$ suddenly changes. 
To be precise, for any $\eps>0$
the following is true.
If $d<1-\eps$ all connected components in this model are of small size \whp, whereas if $d>1+\eps$ there is a single giant component of linear size. 
In
\cite{Karp} Karp gave an elegant proof of 
the original result of \Erdos~and \Renyi~by means of 
branching processes.
A key observation in this work is that the limit of the normalized size of the giant component coincides with the survival
probability of a certain Galton Watson process \cite{Karp}. Though this
Galton Watson tree is the limit of the ''local`` structure of the sparse random graph, this branching process approach allows to describe a invariant of the random graph that highly
depends on its ''global`` structure.

By now, there is a considerable body of work on the component structure in the generalization of the binomial random graph model to hypergraphs. 
In \cite{Pruzan} Schmidt-Pruzan and
Shamir proved a threshold phenomenon similar to the one in \cite{ER} for the random $r$-uniform hypergraph with $r\geq 3$. Furthermore, Behrisch, Coja-Oghlan and Kang presented a probabilistic approach
to determine a local limit theorem for the number of vertices in the largest component of the sparse random hypergraph, \cite{Behrisch,Behrisch2}. This result 
has been complemented 
by \Bollobas~and Riordan 
in \cite{BollobasRiordan}.

One possible way to extend the notion of connected components is to consider subgraphs or subhypergraphs of some given minimal degree. Along these lines the $k$-core of a graph is 
defined as the maximal induced subgraph of minimal degree $k$. In the notion of hypergraphs the degree of a vertex is defined as the number of edges incident to it. Using this definition the 
$k$-core of a hypergraph is defined analogously. The $k$-core of a graph or hypergraph is unique because the union of two subgraphs or subhypergraphs with minimal degree $k$ again has minimal degree
$k$. If there is a vertex of degree strictly less than $k$, removing this vertex and the edges incident to it will not change the $k$-core. Therefore, the $k$-core is obtained 
by repeatedly removing edges incident to
vertices of degree less than $k$ from the graph or hypergraph.  

The $k$-core problem in random graphs has attracted a great deal of attention. 
In \cite{Pittel} Pittel, Spencer and Wormald 
proved that for any $k\geq 3$ the property of having a extensive $k$-core has a sharp threshold and 
determined the precise value below which the $k$-core of the sparse random graph is empty \whp~ In the preliminaries they provided a short heuristic argument to predict this threshold.
What is more, this argument has been used to provide shorter proofs of their result \cite{MolloyCores,Riordan}.
Some structural results on the local structure of the $k$-core follow from either the study of the peeling process \cite{Cooper,JansonLuczak,Kim} or the branching process 
heuristic \cite{Darling,Riordan,Encores}.

Cores in random hypergraphs, where edges consist of $r\geq 3$
vertices, are of further interest because they play an important role in the study of random constraint satisfaction problems, 
\cite{MolloyAchlioptas,Mossel,MolloyRestrepo}. The $k$-core in particular has application to the XORSAT-problem. 
To be precise, every instance of XORSAT
admits a natural hypergraph representation
and
the original instance is solvable iff the system of
linear equations restricted to the $2$-core
of this representation is solvable. 

Motivated by this fact many authors studied the $k$-core problem in $r$-uniform
hypergraphs for $r\geq 3$. The heuristic argument from \cite{Pittel} has been extended by Molloy to determine the threshold $d_{r,k}$ for the appearance of
an extensive $k$-core in the sparse random hypergraph $\vH$. To state his result we denote by $\cC_k(H)$ the set of vertices in the $k$-core of a hypergraph $H$. 
Then Molloy proved that for any $r\geq 3$ and $k\geq 2$ there is a function $\psi_{r,k}:(0,\infty)\to [0,1]$ such that for any $d\in(0,\infty)\setminus\cbc{d_{r,k}}$
the sequence $(n^{-1}|\cC_k(\vH)|)_n$ converges to $\psi_{r,k}(d)$ in probability. The function $\psi_{r,k}$ is identical to $0$ for all $d<d_{r,k}$ and strictly positive for $d>d_{r,k}$.
That is,  
if $d<d_{r,k}$ the $k$-core of $\vH$ is empty \whp, whereas if $d>d_{r,k}$
the $k$-core of $\vH$ contains a linear number of vertices.
Throughout the paper
we will refer to the number $|\cC_k(H)|$ of vertices in the $k$-core as the size of the $k$-core. Beside \cite{MolloyCores} there are several different approaches to determine this threshold by 
studying the peeling process similar to the one in the paper by Pittel, Spencer and Wormald \cite{Encores, Cooper, JansonLuczak, Kim}. In \cite{Botelho} Botelho, Wormald and Ziviani proved
that this threshold remains the same if the model is restricted to $r$-partite $r$-uniform hypergraphs. 
 
The aim of the present paper is to describe the connections between the vertices in the core and in the mantle. More precisely, our aim is to determine the following
distribution. Fix $r\geq 3$, $k\geq 2$, $d>d_{r,k}$.
Generate a random hypergraph $\vH$. Now, apply the peeling process that identifies the $k$-core of $\vH$ and colour each vertex that belongs to the $k$-core black and all other vertices white.
Finally, pick a vertex $\vec v$ uniformly at random.
What is the distribution of the {\em coloured} subhypergraph
induced on the set of all vertices up to some constant distance from $\vec v$?

For simplicity we adopt the factor graph notion that provides 
a natural representation of hypergraphs as bipartite graphs \cite{MM}. In this framework we will introduce 
a branching process that generalizes the usual Galton Watson convergence of the sparse random graph to hypergraphs. This process is the limiting object of the local structure
of the random hypergraph. The appropriate formalism to state this result is the notion of local weak convergence, \cite{Aldous,BenjaminiSchramm,BordenaveCaputo}. 
To complement this process that describes the local structure of the random hypergraph with colours that indicate the membership of the $k$-core
we will characterize the $k$-core by means of a message passing algorithm, called Warning Propagation. 
Just like cores, message passing algorithms are an important tool in the study of random constraint satisfaction problems,
particularly in the context of the study of the ''solution space``~\cite{Barriers,MM,Molloy}.
The characterization of the $k$-core by Warning Propagation was first suggested 
in the physics literature, \cite[\Chap~18]{MM}.
Furthermore, in a recent paper we 
combined this characterization with 
local weak convergence results
to study the connections
of the $k$-core and the mantle 
in the binomial random graph 
\cite{kcore}.
The present paper provides a generalization 
of \cite{kcore} to random hypergraphs.

\subsection{Results}\label{sec_results}
To state our main result it will be convenient to use a natural representation of hypergraphs as bipartite factor graphs. 
We begin with introducing various standard notions for graphs which carry over to this
representation.
In this paper all graphs are assumed to be locally finite with a countable vertex set. Given a vertex $v$ in a graph $G$ we denote by $\partial_G v$ the set of vertices that are adjacent to $v$ 
in $G$. We omit $G$ whenever $G$ is clear from the context. By a rooted graph $(G,v_0)$ we mean a graph $G$ together with a distinguished vertex $v_0$, the root. 
If $\cX$ is a finite set then a $\cX$-marked graph $(G,v_0,\sigma)$ is a rooted graph together with a map $\sigma: V(G)\to\cX$. A bijection $\pi:V(G)\to V(G')$ is called isomorphism if
every edge in $G$ is mapped to an edge in $G'$ and every non-edge in $G$ is mapped to a non-edge in $G'$. Two rooted $\cX$-marked graphs
$(G,v_0,\sigma)$ and $(G',v_0',\sigma')$ are called isomorphic if there is an isomorphism $\pi$ such that $\pi(v_0)=v_0'$
and $\sigma=\sigma'\circ\pi$. We denote by $[G,v_0,\sigma]$ the isomorphism class of $(G,v_0,\sigma)$
and let $\fG^{\cX}$ be the set of all isomorphism classes of $\cX$-marked graphs. 
Finally, for $s\geq 0$ let $\partial^s[G,v_0,\sigma]$ signify the isomorphism class  
of the rooted $\cX$-marked graph obtained by deleting all vertices at a distance greater than $s$ from $v$.

A factor graph 
$F=(\cV(F),\cF(F),\cE(F))$ is a bipartite graph 
on the parts $\cV(F)$ and $\cF(F)$
and edgeset $\cE(F).$
 We will refer to $\cV(F)$ as the variable nodes of $F$ and to $\cF(F)$ as the factor nodes of $F$. 
In the special case of factor graphs we will require that the root of a rooted factor graph is a variable node, i.e. $v_0\in\cV(F).$ 
Similarly, we denote by a $\cX$-marked factor graph a rooted factor graph together with a map $\sigma : \cV(F)\cup\cF(F) \rightarrow \cX$.
The restriction that $v_0\in\cV(F)$ immediately implies that if two rooted factor graphs $(F,v_0,\sigma)$ and $(F',v_0',\sigma')$ are isomorphic
and if 
$\pi:\cV(F)\cup\cF(F)\rightarrow \cV(F')\cup\cF(F')$ is an isomorphism of $F$ and $F'$
then $\pi(\cV(F))=\cV(F')$ and $\pi(\cF(F))=\cF(F')$. We denote by $\fF^{\cX}$ the set of all isomorphism classes of $\cX$-marked
factor graphs.

To represent a given hypergraph $H$
as a factor graph $F_H$, we let $\cV(F_H)$ be the set of vertices in $H$ and $\cF(F_H)$ be the set of edges in $H$. For every variable node $v\in\cV(F_H)$ and factor node 
$a\in\cF(F_H)$ we place an edge in $F_H$ iff $a$ is an edge in $H$ incident to $v$. Then $\partial_{F_H} a$ is nothing but the set of vertices $v$ in $H$ that forms the edge $a$
and $\partial_{F_H}v$ is the set of edges incident to $v$.
Resembling the terminology for hypergraphs, we call a factor graph $F$ $r$-uniform, if all
factor nodes $a\in\cF(F)$ have degree $r$ in $F$, i.e. $|\partial_Fa|=r$. 

Our aim is to study the random $r$-uniform hypergraph together with the marking that indicates the membership of vertices of the $k$-core of $\vH$. 
Therefore, we introduce the notion of $k$-core for $r$-uniform factor graphs. 
In order to match the definition of the 
$k$-core of hypergraphs we will describe the $k$-core of a $r$-uniform factor graph as the subgraph that remains when we repeatedly remove 
factor nodes incident to at least one variable node of degree strictly less than $k$.
We denote by $\cC_k(F)$ the set of variable nodes in the $k$-core of $F$. By the peeling process description it is easy to see that a factor node is in the $k$-core of $F$ iff all its neighbours
are in the $k$-core of $F$. Therefore, the $k$-core of $F$ coincides with the maximal induced subgraph of $F$
such that each variable node $v$ has minimal degree $k$
and each factor node has degree precisely $r$.
If we are given a hypergraph $H$ the peeling process above is analogous to the one described in \S~\ref{sec_intro}, i.e.
a vertex of a hypergraph $H$ is in the $k$-core of $H$ iff the corresponding variable node is in the $k$-core of $F_H$. Let
$\sigma_{k,F}:\cV(F)\cup\cF(F)\rightarrow\{0,1\}$, $v\mapsto\vecone\cbc{v\in\cC_k(F)}$ if $v\in\cV(F)$ and $a\mapsto\vecone\cbc{\partial a\subset\cC_k(F)}$
if $a\in\cF(F)$. 
Finally, denote by $F_x$ the component of $x\in\cV(F)\cup\cF(F)$ in $F$.
In the case of the random $r$-uniform hypergraph we will use the shorthand $F_{\vH}=\vF$ and $\vF_v=F_{\vH,v}.$
Let $\vec v$ denote a randomly chosen vertex in $V(\vH)$. Then our aim is to determine the limiting distribution of $\partial^s[\vF_{\vec v},\vec v,\sigma_{k,\vF_{\vec v}}]$
in the formalism of local weak convergence.

To this end, we construct a multi-type branching process that generates an acyclic factor graph that proceeds in two phases. 
Starting with a single root, which is a variable node, the levels of the process alternately consist of variable and factor nodes. 
This process exhibits $9$ node types in total. The nodes at even distances from the root are variable nodes of possible types $\vertexl$, the
nodes at odd distances are factor nodes of possible types $\edgel$. 
 
The branching process $\hat \T(d,r,k,p)$ has parameters $d, r, k$ and a further parameter $p\in[0,1]$.
Setting $c=c(d,r)=d/(r-1)!$ and
	\begin{align}\label{type_prob}
	q=q(d,k,p)&=\pr\brk{\Po(cp^{r-1}) =k-1|\Po(cp^{r-1})\geq k-1},
	\end{align}
we define
	\begin{align*}
	p_{000}&=1-p,&
	p_{010}&=pq,&
	p_{111}&=p(1-q).
	\end{align*}
Finally, the branching process involves two other coefficients
$$\bar{q}=\bar{q}(d,k,p) =\pr\brk{\Po(cp^{r-1})=k-2|\Po(cp^{r-1})\leq k-2}$$
and
$$\tilde{q}=\tilde{q}(d,k,p) =\pr\brk{\Bin\br{r-1,p}=r-2|\Bin\br{r-1,p}\leq r-2}.$$

The process starts with a single variable node $v_0$, the root. The root is of type $000$, $010$ or $111$ with probability $p_{000}$, $p_{010}$ and $p_{111}$ respectively.
Starting at $v_0$, every variable node has a random number of factor node children independently. The offspring distributions of variable nodes are defined by the generating 
functions in Figure~\ref{Fig_g} where $\vec x=(x_{000},x_{001},x_{010},x_{111})$. That is, 
for an integer vector $\vec y=(y_{000},y_{001},y_{010},y_{111})$ the probability that a variable node of type
$v_1v_2v_3$ generates $y_{a_1a_2a_3}$ children of type $a_1a_2a_3$ for $a_1a_2a_3\in\edgel$ equals the coefficient of 
$x_{000}^{y_{000}}\cdots x_{111}^{y_{111}}$ in $g_{v_1v_2v_3}(\vec x)$. The offspring distributions of factor nodes follow in the same manner from Figure~\ref{Fig_f}, where
$\vec z=(z_{000},z_{001},z_{010},z_{110},z_{111}).$

\begin{figure}\footnotesize
	\begin{align*}
	g_{000}(\vec x)
		&=\exp(cp^{r-1}(x_{000}-1))\frac{\sum_{h=0}^{k-2}(cp^{r-1}x_{010})^{h}/h!}{\sum_{h=0}^{k-2}(cp^{r-1})^h/h!},\\
	g_{001}(\vec x)
		&=\bar{q}\br{\exp(cp^{r-1}(x_{001}-1))x_{010}^{k-2}}\\
		&\quad +(1-\bar{q})\br{\exp(cp^{r-1}(x_{000}-1))\frac{\sum_{h=0}^{k-3}(cp^{r-1}x_{010})^{h}/h!}{\sum_{h=0}^{k-3}(cp^{r-1})^h/h!}},\\
	g_{010}(\vec x)&=\exp(cp^{r-1}(x_{001}-1))x_{010}^{k-1},\\
	g_{110}(\vec x)&=\exp(cp^{r-1}(x_{001}-1))	\frac{\sum_{h\geq k}(cp^{r-1}x_{111})^{h}/h!}{\sum_{h\geq k}(cp^{r-1})^h/h!},\\
	g_{111}(\vec x)&=\exp(cp^{r-1}(x_{001}-1))\frac{\sum_{h\geq k-1}(cp^{r-1}x_{111})^{h}/h!}{\sum_{h\geq k-1}(cp^{r-1})^h/h!}.
	\end{align*}
\caption{The generating functions $g_{v_1v_2v_3}(\vec x)$ of variable nodes.}\label{Fig_g}
\end{figure}
\begin{figure}\footnotesize
	\begin{align*}
	f_{000}(\vec z)
		&=\frac{\sum_{s=0}^{r-2}\bink{r-1}{s}p^s(1-p)^{r-1-s}(qz_{010}+(1-q)z_{110})^{s}z_{000}^{r-1-s}}
                 {\sum_{h=0}^{r-2}\bink{r-1}{s}p^s(1-p)^{r-1-s}},\\
	f_{001}(\vec z)
		&=\tilde{q}\br{z_{001}\br{qz_{010}+(1-q)z_{110}}^{r-2}}\\
		&\qquad+(1-\tilde{q})\frac{\sum_{s=0}^{r-3}\bink{r-1}{s}p^s(1-p)^{r-1-s}(qz_{010}+(1-q)z_{110})^{s}z_{000}^{r-1-s}}
                 {\sum_{h=0}^{r-3}\bink{r-1}{h}p^h(1-p)^{r-1-h}},\\
	f_{010}(\vec z)&=\br{qz_{010}+(1-q)z_{110}}^{r-1},\\
	f_{111}(\vec z)&=z_{111}^{r-1}.
	\end{align*}
\caption{The generating functions $f_{a_1a_2a_3}(\vec z)$ of factor nodes.}\label{Fig_f}
\end{figure}
Let $\T(d,r,k,p)$ denote the $\{0,1\}$-marked acyclic factor graph by giving mark $0$ to 
all nodes of type $000$, $001$ or $010$, and mark $1$ to vertices of type $110$ or $111$. 
Then $\T(d,r,p)$ has $4$ different types of nodes in total. Recall that $d_{r,k}$ is the threshold for the appearance of a extensive $k$-core 
in $\vH$. Our main result is the following.

\begin{theorem}\label{Thm_main}
Assume that $r\geq 3$, $k\geq 2$ and $d>d_{r,k}$.
 Moreover, let $p^*$ be the largest fixed point of the continuous function
 	\begin{equation}\label{eqmain}
 	\phi_{d,r,k}:[0,1]\to[0,1],\qquad p\mapsto\pr\brk{\Po\br{\frac{d}{(r-1)!}p^{r-1}}\geq k-1}.
 	\end{equation}
If $\tau$ is a rooted acyclic factor graph with marks in $\{0,1\}$, then 
	$$\frac1n\sum_{v\in \cV(\vF)}
		\vecone\cbc{\partial^\omega[\vF,v,\sigma_{k,\vF_{v}}]=\partial^\omega[\tau]}$$
converges to $\pr\brk{\partial^\omega[\T(d,r,k,p^*)]=\partial^\omega[\tau]}$ in probability.
\end{theorem}

\Thm~\ref{Thm_main} describes the limiting behaviour of the local structure of the random factor graph $\vF$.
We will reformulate this result in the framework of local weak convergence,
closely following~\cite{BordenaveCaputo},
which is part of weak convergence theory of probability measures.
To introduce this formalism
we assume that $\{0,1\}$ is equipped 
with the discrete topology.
Further, we endow $\fG^\cX$ with the coarsest topology that makes all the maps
\begin{equation}\label{topofct}
\kappa_{G,s}:\fG^\cX\rightarrow \{0,1\},~G'\mapsto\vecone\cbc{\partial^sG'\ism G} \mbox{ for } G\in\fG^\cX,~s\geq 0
\end{equation}
continuous.   
Let $\cP(\fG^{\cX})$ and $\cP^2(\fG^{\cX})$ denote the spaces of probability measures on $\fG^{\cX}$ and $\cP(\fG^{\cX})$, both equipped with the topology induced
by weak convergence of probability 
measures. 
For $G\in\fG^{\cX}$ let $\atom_{G}\in\cP(\fG^{\cX})$ signify the Dirac measure on $G$.
For $\lambda\in\cP(\fG^{\cX})$ let $\atom_\lambda\in\cP^2(\fG^{\cX})$ be the Dirac measure on $\lambda$. We reformulate our result in 
$\cP(\fF^{\cX})\subset\cP(\fG^{\cX})$ and $\cP^2(\fF^{\cX})\subset\cP^2(\fG^{\cX})$ as follows.

Given a $r$-uniform factor graph
on $n$ variable nodes let
$$\lambda_{k,F}=\frac{1}{n}\sum_{v\in\cV(F)}\atom_{[F_v,v,\sigma_{k,F}]}\in\cP(\fF^{\{0,1\}})$$
be the empirical distribution of the neighbourhoods of variable nodes marked according to the  membership of the 
$k$-core
and define
\begin{equation}\label{eqLambdadkn}
  \Lambda_{d,r,k,n}=\Erw_{\vF}\brk{\atom_{\lambda_{k,\vF}}}\in\cP^2(\fF^{\{0,1\}}).
\end{equation}
Finally, let 
$$\thet_{d,r,k,p}=\cL[\T(d,r,k,p)]\in\cP(\fF^{\{0,1\}}).$$
Then we can reformulate our result as follows.
\begin{theorem}\label{Thm_lwc}
Assume that $r \geq 3$, $k\geq 2$ and $d>d_{r,k}$. Let $p^*$ be the largest fixed point of the function  $\phi_{d,r,k}$ from (\ref{eqmain}).
Then $\lim_{n\to\infty}\Lambda_{d,r,k,n}=\atom_{\thet_{d,r,k,p^*}}$.
\end{theorem}

In \S~\ref{sec_mainproof} we shall derive  \Thm~\ref{Thm_main} from \Thm~\ref{Thm_lwc}.

\subsection{Related work}
Some structural results on the $k$-core problem
in random hypergraphs follow
from previous analysis, in particular
 from the study of random constraint satisfaction problems.
In \cite{Dembo}
 Montanari and Dembo determined the fluctuation in the size of
the $k$-core near the threshold $d_{r,k}$. Furthermore, they determined the probability that $\vH$ has a non-empty $k$-core as $d$ approaches $d_{r,k}.$  In \cite{MolloyAchlioptas} Molloy and Achlioptas studied the peeling process that gives the $k$-core and
proved that every vertex that is not in the $k$-core can be removed by a sequence of at most $O(\log n)$ steps of the peeling process.
The Poisson-Cloning model \cite{Kim} enables to describe the local
internal structure of the $k$-core in the
sparse random hypergraph by a 
simple Galton Watson process.
In \cite{Cooper} Cooper studied the $k$-core of random hypergraphs with some given degree 
sequence and bounded maximal degree. Cooper derived the asymptotic distribution of vertices with a given number of neighbours in the $k$-core and in the mantle.

In a recent paper we used the characterization of the $k$-core by means 
of Warning Propagation to
study the $k$-core problem in sparse random graphs \cite{kcore}. The proof strategy
of the present paper is to transfer 
the analysis from \cite{kcore} to the 
factor graph representation of 
hypergraphs.
The essential difference in the graph case 
is that in the Warning Propagation
description there is no exchange of information between vertices and edges in the graph. This is illustrated in the fact that there is only one type of Warning Propagation messages in that case. The factor graph notion enables us to describe such communication and adopt ideas 
from \cite{kcore} to the hypergraph 
setting. 
In the hypergraph case, the fact that the Warning Propagation
description needs two distinct types of  messages is expressed in the two phases of the branching process in our main result, see \Thm~\ref{Thm_main},
Figure~\ref{Fig_g} and Figure~\ref{Fig_f}.

In the context of the study of random constraint satisfaction problems, a similar approach has been used
in earlier work on the $k$-core 
problem in hypergraphs.
In~\cite[proof of \Lem~6]{Molloy} Molloy considers a deletion process on the random hypergraph which corresponds to the Warning Propagation
description of the $k$-core. His analysis to determine the $k$-core threshold leads him to the study of a limit 
similar to~\Lem~\ref{Lemma_as}. 
Furthermore, Ibrahimi,
Kanoria, Kraning and Montanari \cite{Ibrahimi} used
the connection of Warning Propagation
and the $k$-core problem 
  to describe the local structure
of the $2$-core in random hypergraphs 
by the limit
of Warning Propagation on a suitable Galton Watson
tree. 
A key technical ingredient
of the present paper is to
further investigate the connection between Warning Propagation and the fixed point problem from \S~\ref{sec_thresh}. This connection
is essential to combine Warning Propagation with the local weak convergence result in order to exhibit a branching process that describes the $k$-core together with the mantle. To be precise, 
by contrast with \cite{Ibrahimi}
we describe the local structure of the
random hypergraph together with the 
colouring that indicates the $k$-core
membership by means of a random tree
that proceeds in a ``top-down"
fashion. In particular, the colours are constructed as the 
tree evolves rather than a posteriori
by applying Warning Propagation to the
random tree. This description can not be
extracted from previous analysis.

\section{Preliminaries}
\subsection{Notation}
For a graph $G$ and a vertex $v\in V(G)$, $\partial_Gv$ denotes the set of all vertices adjacent to $v$ in $G$
and $G_v$ denotes the component 
of $v$ in $G$.
Furthermore, we denote 
by $\partial^s(G,v)$ the subgraph of $G$ induced on the set of vertices at distance at most $s$ from $v$. 
We abbreviate $\partial (G,v) = \partial^1(G,v)$ and denote by $\partial v$ the neighbourhood of $v$ in $G$ whenever $G$ is clear from the context.

By a {\em rooted graph} we mean a connected locally finite graph $G$ together with a distinguished root $v_0\in V(G)$.
A {\em child} of $v\in V(G)$ is a vertex $u\in\partial v$ whose distance from $v_0$ is greater than that of $v$.
For a vertex $v$ of a rooted graph $G$ we denote by $\partial_+v$ the set of all children of $v$.

We say that $(G,v_0)$ and $(G',v_0')$ are isomorphic if there is an isomorphism $\pi:G\to G'$ 
such that $\pi(v_0)=\pi(v_0')$.
We denote by $[G,v_0]$ the isomorphism class of $(G,v_0)$. Furthermore, we let $\fG$ be the set of all isomorphism classes of rooted graphs.
Finally, let $\partial^s[G,v_0]$ be the isomorphism class of $\partial^s(G,v_0)$.
For factor graphs we will assume that $v_0\in\cV(F)$ and denote 
by $\fF$ the set of all isomorphism 
classes of rooted factorgraphs.

If $X$ is a random variable
on a space $\cE$, we let $\cL(X)$ denote the distribution of $X$. That is, 
$\cL(X)$ is a probability distribution 
on $\cE.$
For real numbers $y,z>0$ we let $\Po_{\geq z}(y)$ denote the Poisson distribution $\Po(y)$ conditioned on the event that it attains a value least $z$.
Thus, 
	$$\pr\brk{\Po_{\geq z}(y)=x}=\frac{\vecone\cbc{x\geq z}y^x\exp(-y)}{x!\pr\brk{\Po(y)\geq z}}\qquad\mbox{for any integer $x\geq0$}.$$
Similarly, we let $\Bin_{\geq z}(m,y)$ denote the binomial distribution $\Bin(m,y)$ conditioned on the event that the outcome is at least $z$. That is,
       $$\pr\brk{\Bin_{\geq z}(m,y)=x}=\frac{\vecone\cbc{x\geq z}\bink{m}{x}y^x(1-y)^{m-x}}{\pr\brk{\Bin(m,y)\geq z}}\qquad\mbox{for any integer $0\leq x\leq m$}.$$
The distributions $\Po_{> z}(y)$, $\Po_{\leq z}(y)$, $\Po_{< z}(y)$,  $\Bin_{> z}(m,y)$, $\Bin_{\leq z}(m,y)$, $\Bin_{< z}(m,y)$ are defined analogously. 
Throughout the paper we will define various probability distributions and random factor graphs. We provide tables of these definitions in the appendix.

\subsection{The local structure} 
In what follows it will be 
convenient to work with random 
factor graphs rather than their isomorphism
classes. In order to  obtain
a representative of the corresponding
isomorphism
class, from now 
on we will assume that any factor 
graph is additionally equipped with 
almost surely distinct random labels 
at variable and factor nodes. 
This labelling  will be technically necessary in upcoming proofs, but
the actual label of a node will not be taken into account in any of them.

For instance, 
recall that $c=c(d,r)=d/(r-1)!$ and consider the following random recursive acyclic factor graph. The process starts with a variable node $v_0$, the root.
Starting at $v_0$ every variable node has a 
$\Po(c)$ number of factor node children independently and each factor node has $r-1$ variable node children deterministically. Let
$\T(d,r)$ 
denote the
random tree which is obtained by labelling each nodes in this process with a number in $[0,1]$ independently and uniformly at random.
In the following we need to prove that $\vF$ ``locally" converges to $[\T(d,r)]$
in the notion of local weak convergence.

To this end, we endow $\fG$ with the coarsest topology that makes all the maps $\kappa_{G,s}:\fG\rightarrow\{0,1\},~G'\mapsto\vecone\cbc{\partial^sG'=\partial^sG}$ for fixed
$s\geq 0, G\in\fG$ continuous. Let $\cP(\fG)$ be the set of all probability measures on $\fG$ together with the weak topology. Finally, let $\cP^2(\fG)$
be the set of all probability measures on $\cP(\fG)$ also together with the weak topology. For a given factor graph $F$ 
on $n$ variable nodes
let
$$\lambda_{F}=\frac{1}{n}\sum_{v\in\cV(F)}\atom_{[F_v,v]}\in\cP(\fF).$$
Further, let
\begin{equation}\label{eqLambdadn}
\Lambda_{d,r,n}=\Erw_{\vF}[\atom_{\lambda_{\vF}}]\in\cP^2(\fF).
\end{equation}

Then the following theorem expresses that ''locally`` $\vF$ converges to $\T(d,r)$. 
\begin{theorem}\label{Thm_GW}
For any $r\geq 3$ and $d>0$, we have $\lim_{n\to\infty}\Lambda_{d,r,n}=\atom_{\cL([\T(d,r)])}$.
\end{theorem} 
We provide a proof of \Thm~\ref{Thm_GW} in the appendix.

\subsection{The $k$-core threshold}\label{sec_thresh}
We build upon the following result that determines the $k$-core threshold and the asymptotic size of the $k$-core. Recall that 
$\vF$ is the random factor graph which is constructed from $\vH$ and that $\cC_k(\vF)$ denotes the set of variable nodes in the $k$-core of $\vF$

\begin{theorem}[{\cite{Molloy}}]\label{Thm_Mol}
Let $r\geq 3$ and $k\geq 2$. Then the function  $\lambda\in(0,\infty)\mapsto\lambda(r-1)!/(\pr\brk{\Po(\lambda)\geq k-1})^{r-1}$ is continuous, 
tends to infinite as $\lambda\to0$ or $\lambda\to\infty$,
and has a unique local minimum, where it attains the value
	$$d_{r,k}=\inf\cbc{\lambda(r-1)!/(\pr\brk{\Po(\lambda)\geq k-1})^{r-1}:\lambda>0}.$$
Furthermore, for any $d>d_{r,k}$ the equation $d=\lambda(r-1)!/(\pr\brk{\Po(\lambda)\geq k-1})^{r-1}$ has precisely two solutions.
Let $\lambda_{r,k}(d)$ denote the larger one and define
	\begin{align*}
	\psi_{r,k}&:(d_{r,k},\infty)\to(0,\infty),\qquad d\mapsto\pr\brk{\Po(\lambda_{r,k}(d))\geq k}.
	\end{align*}
Then $\psi_{r,k}$ is a strictly increasing continuous function.
Moreover, if $d<d_{r,k}$, then $\cC_k(\vF)=\emptyset$ \whp,
while $\frac1n|\cC_k(\vF)|$ converges to $\psi_{r,k}(d)$ in probability for $d>d_{r,k}$.
\end{theorem}

The following lemma establishes a connection between Warning Propagation and the formula for the size of the $k$-core from \Thm~\ref{Thm_Mol}.
Similar statements are implicit in~\cite{MolloyCores,Riordan}. Nevertheless, we provide a simple proof similar to the one in 
\cite{kcore} to keep the present paper self-contained.

\begin{lemma}\label{Prop_fix}
Suppose $d>d_{r,k}$
and let $p^*$ be the largest fixed point of 
the function $\phi_{d,r,k}$ from (\ref{eqmain}).
Then $\phi_{d,r,k}$ is contracting on the interval $[p^*,1]$.
Moreover, $\psi_k(d)= \phi_{d,r,k+1}(p^*).$
\end{lemma}
\begin{proof}

Let 
$\varphi_k(x)=
\pr \brk{\Po(x)\ge k-1}$
for $x\geq 0$. Then 
$\phi_{d,r,k}(p)=\varphi_k\br{\frac{d}{(r-1)!}p^{r-1}}.$
Moreover, 
\begin{align*}
 d=\frac{\lambda(r-1)!}{\varphi_k(\lambda)^{r-1}} \quad\mbox{iff }\quad\phi_{d,r,k}\br{\br{\frac{\lambda(r-1)!}{d}}^{1/(r-1)}}=\br{\frac{\lambda(r-1)!}{d}}^{1/(r-1)}
\end{align*}
i.e. $\lambda$ is a solution to the left equation iff $(\lambda(r-1)!/d)^{1/(r-1)}$ is a fixed point of $\phi_{d,r,k}$.
Since $p^*=p^*(d,r,k)$ is the largest fixed point of $\phi_{d,r,k}$, it holds that
$p^*=(\lambda_{r,k}(d)(r-1)!/d)^{1/(r-1)}$, where $\lambda_{r,k}(d)$ is as in \Thm~\ref{Thm_Mol}. By  \Thm~\ref{Thm_Mol} this implies that
$$\psi_{r,k}(d)=\pr\br{\Po(\lambda_{r,k}(d))\geq k}=\pr\br{\Po(d/(r-1)!{p^*}^{r-1}\geq k} = \phi_{d,r,k+1}\br{p^*}.$$

Our aim is to prove that $\phi_{d,r,k}$ is contracting on the interval $[p^*,1].$
To this end, we begin with showing that $p^*(d,r,k)>0$ if $d>d_{r,k}$. By Theorem~\ref{Thm_Mol},  we have
$ d_{r,k} = \inf \left\{\left.\lambda/((r-1)!\varphi_k(\lambda)^{r-1})\right|\lambda>0\right\}.$
Furthermore, $\lambda_{r,k}(d_{r,k}) \geq 0$ by definition. Since $\lambda_{r,k}(d)$ is a strictly increasing in $d$, it holds that 
$\lambda_{r,k}(d)>\lambda_{r,k}(d_{r,k}) \geq 0 $. Using $p^*=(\lambda_{r,k}(d)(r-1)!/d)^{1/(r-1)}$ this implies
$p^*=p^*(d,r,k)>0$ for $r\geq 3$, $k\geq 2$ and $d>d_{r,k}$.

Finally, it remains to prove that $\phi_{d,r,k}$ is a contraction on $[p^*,1]$. The derivatives of $\varphi_k$ read as follows.
 	\begin{eqnarray*}
 	\frac{\partial}{\partial x}\varphi_k(x)&=&\frac{1}{(k-2)!\exp(x)}x^{k-2}\geq0,\\
	\frac{\partial^2}{\partial x^2}\varphi_k(x)&=&\frac{1}{(k-1)!\exp(x)}(-x+k-2)x^{k-3}.
	\end{eqnarray*}
Using
$$\frac{\partial}{\partial p}\phi_{d,k,r}(p)=\frac{d}{(r-2)!}p^{r-2}\frac{\partial}{\partial x}\left.\varphi_k(x)\right|_{x=\frac{dp^{r-1}}{(r-1)!}}$$
and 
$$\frac{\partial^2}{\partial p^2 }\phi_{d,k,r}(p)=\frac{d}{(r-3)!}p^{r-3}\frac{\partial}{\partial x}
\left.\varphi_k(x)\right|_{x=\frac{dp^{r-1}}{(r-1)!}} + \frac{d^2}{((r-2)!)^2}p^{2r-4}\frac{\partial^2}{\partial x^2}
\left.\varphi_k(x)\right|_{x=\frac{dp^{r-1}}{(r-1)!}}$$
we obtain $\frac{\partial}{\partial p}\phi_{d,r,k}(p)\geq 0$ for all $p\in[0,1]$ and $\frac{\partial^2}{\partial p^2}\phi_{d,r,k}(p)=0$ 
iff $p=p_0=((2rk-3r-3k+4)(r-2)!/d)^{1/(r-1)}.$
Hence, $\partial^2/\partial p^2\phi_{d,r,k}$ has only one root in $[0,1]$ and $\phi_{d,r,k}$ is convex on $[0,p_0]$ 
and concave on $(p_0,1]$. 
This means that apart from the trivial fixed point at $p=0$ the function $\phi_{d,r,k}$ either has one additional fixed point
$p^*>0$ or two additional fixed points $p^*>p_1>0.$ In total, three possible cases can occur. The first case is that $\phi_{d,r,k}$ has only one additional fixed point
and is tangent to the identity at $p^*$. The second case is that $\phi_{d,r,k}$ has only one additional fixed point, but has 
derivative at least one at the point $p=0$.
Then $\phi_{d,r,k}$ crosses the identity function at the point $p^*$.
Finally, $\phi_{d,r,k}$ can have two additional fixed points 
$p^*>p_1>0$, where $\phi_{d,r,k}$ crosses the identity.
Our aim is to prove that the first case can not occur if $p^*>0.$
Therefore, suppose it occurs for some $d'$, and let $d''$ be such that $d_{r,k} < d'' < d'$. Then 
$\phi_{d'',r,k}(p)< \phi_{d',r,k}(p) \le p$
for all $p>0$. This means that $p^*(d'',r,k)=0$, which contradicts the first part of the proof.
Therefore, it holds that $\phi_{d,r,k}$ crosses the identity at $p^*>0.$
Then it follows from $\phi_{d,r,k}(1)<1$ and the fact that 
$\partial^2\phi_{d,r,k}/\partial p^2$ changes its sign only once in $[0,1]$, that $\phi_{d,r,k}$ is concave on $[p^*,1].$ 
Furthermore, by $\frac{\partial}{\partial p}\phi_{d,r,k}(p)\geq 0$ the function
$\phi_{d,r,k}$ is monotonically increasing on $[p^*,1]$. Therefore, its derivative is less than one on this interval and the assertion follows.
\qquad
\end{proof}

\section{Warning Propagation and the $k$-core}
Suppose that $F$ is a given input factor graph which is $r$-uniform. To identify the $k$-core of $F$ we adopt the Warning Propagation message passing algorithm, \cite{MM}. Warning Propagation
is based on introducing messages on the edges of $F$ and marks on the nodes of $F$.
Messages and marks have values in $\{0,1\}$.
Both will be updated in terms of a time parameter $t\geq 0$. The messages are directed objects. That is for every
pair of adjacent nodes $v\in\cV(F)$ and $a\in\cF(F)$ and any time $t\geq 0$ there is a message $\mu_{v\to a}(t|F)$ 
``from $v$ to $a$'' and a message $\mu_{a\to v}(t|F)$ in the opposite direction. Similarly, we distinguish between the marks $\mu_{a}(t|F)$ of factor nodes $a$ and the marks
$\mu_{v}(t|F)$ of variable nodes $v$. 
At time $t=0$ we start with the configuration in which all messages in both directions are equal to $1$, i.e.
 writing $\partial x$ for the set of neighbours of $x$ in $F$, we let
	 \begin{equation}\label{vinitialisation}
	    \mu_{v\to a}(0|F)=\mu_{a\to v}(0|F)=1
	 \end{equation}
for all $a\in \cF(F)$ and $v\in\partial a$.
Inductively, abbreviating $\partial x\setminus y=\partial x\setminus\cbc y$, we let
	\begin{equation}\label{vmessages}
	  \mu_{v\to a}(t+1|F)=\vecone\cbc{\sum_{b\in\partial v\setminus a}\mu_{b\to v}(t|F)\geq k-1}
	\end{equation}
and
	\begin{equation}\label{emessages}
	  \mu_{a\to v}(t+1|F)=\vecone\cbc{\sum_{w\in\partial a \setminus v}\mu_{w\to a}(t|F)=r-1}.
	\end{equation}
Additionally, the mark of $v\in\cV(F)$ at time $t$ is
	\begin{equation}\label{vmarks}
	  \mu_v(t|F)=\vecone\cbc{\sum_{a\in\partial v}\mu_{a\to v}(t|F)\geq k}.
	\end{equation}
The mark of $a\in\cF(F)$ at time $t$ is
	\begin{equation}\label{emarks}
	  \mu_a(t|F)=\vecone\cbc{\sum_{v\in\partial a}\mu_{v\to a}(t|F)=r}.
	\end{equation}
Note that the Warning Propagation description is similar to the peeling process introduced in \S~\ref{sec_intro}. 
At time $t=1$ every variable node $v$ of degree strictly less than $k$ will send message $\mu_{v\to a}(1|F)=0$ to all its neighbours $a$. As a consequence,
 at time $t=1$
we set the mark of every factor node $a$ adjacent 
to $v$ to $\mu_{a}(1|F)=0$. 
In the subsequent iteration
the mark of every variable node 
$w$ that is adjacent to less than $k$ 
factor nodes marked with $1$ will be set to
$\mu_{w}(1|F)=0.$
If we interpret a mark equal to $0$
as deletion from the factor graph,
at each iteration Warning Propagation simultaneously deletes all factor nodes 
adjacent to variable nodes of degree less that $k$ from $F$. 
Along these lines, the intuition is that if $\mu_x(t|F)=0$ for $t\geq 0$, then after $t$ iterations Warning Propagation has ruled out that $x$ belongs to the $k$-core of $F$. 

\subsection{Convergence to the $k$-core}
The aim of this section is to reduce the study of the $k$-core of $\vF$ to the analysis of Warning Propagation on the random tree
$\T(d,r)$ that describes the local structure of $\vF$. 
The following proof strongly resembles our prior analysis in \cite{kcore} and the study of Warning Propagation in the context of 
the $k$-XORSAT problem (cf.~\cite[Chapter~19]{MM}). 

\begin{lemma}\label{Lemma_WP}
Let $F$ be a locally finite $r$-uniform factor graph.
\begin{itemize}
\item[(1)] All messages and marks of Warning Propagation on $F$ are monotonically decreasing in $t$.
\item[(2)] For any $t\geq0$ we have $\core_k(F)\subset\cbc{v\in \cV(F):\mu_v(t|F)=1}$.
\item[(3)] For any node $x$, $\lim_{t\to\infty}\mu_x(t|F)$ exists
	and node $x$ belongs to the $k$-core of $F$ iff \linebreak[4]
	$\lim_{t\to\infty}\mu_x(t|F)=1$.
\end{itemize}
\end{lemma}
\begin{proof}
To obtain (1), we  show by induction that, in fact, for any $v$ and $a\in\partial v$
\begin{align*}
A^{(1)}(t): & \hspace{0.2cm} \mu_{v\to a}(t+1|G)\leq\mu_{v\to a}(t|G)\\
B^{(1)}(t): & \hspace{0.2cm} \mu_{a\to v}(t+1|G)\leq\mu_{a\to v}(t|G)
\end{align*}
for all $t\ge 0$. Properties $A^{(1)}(0)$ and $B^{(1)}(0)$ are certainly true because of our starting conditions (\ref{vinitialisation}). 
Properties $A^{(1)}(t+1)$ and $B^{(1)}(t+1)$ follow from $A^{(1)}(t)\wedge B^{(1)}(t) \Rightarrow A^{(1)}(t+1)\wedge B^{(1)}(t+1)$ with 
(\ref{vmessages}) and (\ref{emessages}). The marks are decreasing in $t$ with (\ref{vmarks}) and (\ref{emarks}).

To obtain (2), we  show by induction that, in fact, any variable node $v$ and factor node $a$ in the core satisfy
\begin{align*}
A^{(2)}(t): & \hspace{0.2cm} \mu_{v\to b}(t|G)=1 \hspace{0.2cm} \mbox{for all } b\in \partial_F v\\
B^{(2)}(t): & \hspace{0.2cm} \mu_v(t|G) =1\\
C^{(2)}(t): & \hspace{0.2cm} \mu_{a\to w}(t|G)=1 \hspace{0.2cm} \mbox{for all } w\in \partial_F a\\
D^{(2)}(t): & \hspace{0.2cm} \mu_a(t|G) =1
\end{align*}
for all $t\geq 0$. Properties $A^{(2)}(0)$ and $C^{(2)}(0)$ are certainly true because of our starting conditions. We will show that for all $t\geq 0$, 
$A^{(2)}(t)\wedge C^{(2)}(t) \Rightarrow A^{(2)}(t+1)\wedge B^{(2)}(t) \wedge C^{(2)}(t+1)\wedge D^{(2)}(t)$. This follows because every variable node $v$
has at least $k$ neighbours in the core, also with $\mu_{a\to v}(t|F)=1 $ if $C^{(2)}(t)$ holds, and so (\ref{vmarks}) implies that 
$B^{(2)}(t)$ holds while (\ref{vmessages}) implies that $A^{(2)}(t+1)$ holds. Now, let $a$ be a factor node in the core, then $a$ has exactly $r$ neighbours
$v$ which are in the core. If $A^{(2)}(t)$ holds, then (\ref{emarks}) and (\ref{emessages}) imply $D^{(2)}(t)$ and $C^{(2)}(t+1)$.      

To obtain (3), let $F'$ be the subgraph of the factor graph induced by the nodes $x$ with
$\lim_{t\to\infty}\mu_x(t|F)=1$. The second part of this lemma implies that $\core_k(F)\subset\cV(F')$. Now, let $v$ be a function node in $F'$. Since 
$F$ is locally finite, it holds that $|\partial v|<\infty$ and so there is a time
$t_0$ such that the messages $\mu_{v\to a}(t|F)$ and $\mu_{a\to v}(t|F)$ remain constant. Since $v\in F'$ it holds that $\mu_v(t_0|F)=1$ and (\ref{vmarks})
implies that $\sum_{a\in\partial v}\mu_{a\to v}(t_0|F)\geq k$ and therefore $\mu_{v\to a}(t_0+1|F)=1$ for all $a\in\partial v$. Now, let $a\in\partial v$ 
with $\mu_{a\to v}(t_0+1|F)=1$, then (\ref{emarks}) implies that $\mu_a(t_0+1|F)=1$. Since the mark at $a$ remains constant after time $t_0$ it holds that 
$u\in F'$ for all $u\in\partial a\setminus v$. Since there are at least $k$ such edges it holds that $\cV(F')\subset\core_k(H).$
\qquad
\end{proof}

\subsection{Warning Propagation and the local structure}
In this section we begin with a analysis of Warning Propagation on the random tree $\T(d,r)$ which closely resembles the 
analysis from 
\cite{kcore}. This first step will enable us to make a connection between Warning Propagation
and the fixed point problem studied in \S~\ref{sec_thresh}.
Assume that $(T,v_0)$ is a rooted locally finite acyclic factor graph.
Then by definition the root $v_0\in\cV(T)$ is a variable node of $T$. Furthermore, starting with $v_0$, $T$ consists of consecutive layers of variable and factor nodes.
Let $x$ be a node of $T$ different from the root, then we refer to the neighbour $y$ of $x$ on the path to $v_0$ as the ''parent`` of $x$.
That is, 
a variable node $v\neq v_0$ has a parent $a\in\cF(T)$ and a factor node $b$ has a parent $w\in\cV(T)$.
We use the shorthand
	$$\mu_{v\toboss}(t|T)=\mu_{v\to a}(t|T)\quad\mbox{and}\quad\mu_{b\toboss}(t|T)=\mu_{b\to w}(t|T).$$
In addition, we let
	$$\mu_{v_0\toboss}(0|T)=1\quad\mbox{and}\quad
		\mu_{v_0\toboss}(t|T)=\vecone\cbc{\sum_{a\in\partial v_0}\mu_{a\to v_0}(t-1|T)\geq k-1}~\mbox{for $t> 0$ }.$$

\noindent
The following Lemma enables us to make a connection between Warning Propagation on 
$\T(d,r)$ and the fixed point problem from \S~\ref{sec_thresh}.

\begin{lemma}\label{Lemma_as}
The sequence $(\mu_{v_0\toboss}(t|\T(d,r)))_{t\geq1}$ converges almost surely to a random variable in $\{0,1\}$
whose expectation is equal to $p^*$. In particular, we have $\Erw[\mu_{v_0\toboss}(t|\T(d,r))]$ converges to $p^*$ as $t$
tends to infinity.

\end{lemma}
\begin{proof}
Let $p^{\bc 0}=1$ and $p^{(t)}=\Erw[\mu_{v_0\toboss}(t|\T(d,r))]$ for $t\geq1$. Then we obtain that $p^{(1)}=\pr\br{\Po(d/(r-1)!)\geq k-1}= \phi_{d,r,k}(1).$

For $t\geq 2$ we use the fact that random tree $\T(d,r)$ is recursive.
Given the degree of the root $v_0$, the messages $\mu_{v\toboss}(t|\T(d,r))$ for variable nodes $v$ at distance $2$ from $v_0$ are 
mutually independent Bernoulli variables with mean $p^{(t)}$. Since each factor node $a\in\partial v_0$ has $r-1$ children, we 
have
$\pr\br{\mu_{a\toboss}(t+1|\T(d,r))=1}={p^{(t)}}^{r-1}.$
Since the degree of $v_0$ is a Poisson variable with mean $d/(r-1)!$, we conclude that

        $$p^{(t+2)}=\pr\br{\Po\br{ \frac{d}{(r-1)!}{p^{(t)}}^{r-1}    }\geq k-1}=\phi_{d,r,k}\br{p^{(t)}}\qquad\mbox{for any }t\geq0.$$

By \Lem~\ref{Prop_fix}, $\phi_{d,r,k}$ is contracting on $[p^*,1]$.
This implies that $\lim_{t\to\infty}p^{(t)}=p^*$.
Furthermore, by \Lem~\ref{Lemma_WP} the sequence $(\mu_{v_0\toboss}(t|\T(d,r)))_{t\geq1}$ is monotonically decreasing. Applying the monotone convergence theorem we conclude that 
$(\mu_{v_0\toboss}(t|\T(d,r)))_{t\geq1}$ converges almost surely.
\qquad
\end{proof}

\subsection{Back to hypergraphs}
Let $t\geq0$. We denote by $\T_{t}(d,r,k)$ the random $\cbc{0,1}$-marked tree which is obtained by marking each node $x$ in $\T(d,r)$ with 
its Warning Propagation mark $\mu_{x}(t|\T(d,r))\in\cbc{0,1}$.
The aim in this section is to prove the following Proposition.

\begin{proposition}\label{Prop_main}
The limits
	\begin{align*}
	\theta_{d,r,k}&:=\lim_{t\to\infty}\cL([\T_t(d,r,k)]),&\hspace{-2cm}\Lambda_{d,r,k}&:=\lim_{n\to\infty}\Lambda_{d,r,k,n}
	\end{align*}
exist and $\Lambda_{d,r,k}=\atom_{\theta_{d,r,k}}$.
\end{proposition}

\Lem~\ref{Lemma_WP} implies that the $k$-core of $F$ is contained in the set of all nodes  which are marked with $1$ after $t$ rounds of Warning Propagation. Additionally, the following lemma 
implies that in the case of the random factor graph $\vF$ a bounded number of iterations is enough to give a good approximation to the $k$-core of $\vF$.

\begin{lemma}\label{Lemma_eps}
 For any $\eps>0$ there is $t>0$ s.t.
	$\abs{\cbc{v\in [n]:\mu_v(t|\vF)=1}\setminus\core_k(\vF)}\leq\eps n$ \whp\
\end{lemma}
\begin{proof}
To prove the above statement let 
$$Y_n=\frac 1n |\cC_k(\vF)| \mbox{ and } X_n(t)=\frac 1n |\{v\in[n]:\mu_v(t|\vF)=1\}|$$
be the fraction of nodes in the $k$-core of $\vF$ and the fraction of 
nodes labelled with  $1$ after $t$ iterations of Warning Propagation respectively. Further, let $x^{(t)}=\Erw_{\T(d,r)}[\mu_{v_0}(t|\T(d,r)].$ By the previous lemma it holds that
$X_n(t)$ converges to $x^{(t)}$ in probability as $n$ tends to infinity. On the other hand, from \Thm~\ref{Thm_Mol} and \Lem~\ref{Prop_fix} we know that $Y_n$ converges to 
$\phi_{d,r,k+1}(p^*)$ in probability as $n$ tends to infinity. Finally, by \Lem~\ref{Lemma_as} it holds that 
\begin{equation}\label{eqConv}
  \Erw_{\T(d,r)}[\mu_{v_0\toboss}(t|\T(d,r)]\to p^*
\end{equation}
deterministically as $t$ tends to infinity. We can apply \eqref{eqConv} to each child $v$ of $v_0$ in $\T(d,r)$. Therefore, we obtain
\begin{align*}
 x^{(t+1)}&=\pr\br{\sum_{v\in\partial v_0} \mu_{v\toboss}(t|\T(d,r))\geq k}\\
          &=\pr\br{\Po\br{dp^{(t)}}\geq k)}
          \to \pr\br{\Po(dp^*)\geq k}
          =\phi_{d,r,k+1}(p^*).
\end{align*}
Hence, for suitable $t$ we conclude
$\pr\br{|Y_n-X_n(t)|\geq\eps}\to 0$ as $n$ tends infinity. 
\qquad
\end{proof}

To prove the main result of this subsection we will first show that the distribution of the neighbourhoods of the random factor graph $\vF$ marked according to $\mu_{\nix}(t|\vF)$ 
is described by the distribution of $\T(d,r)$ together with Warning Propagation after $t$ rounds. Recall that
$\cP(\fG^{\cX})$ and $\cP^2(\fG^{\cX})$ denote the spaces of probability measures on $\fG^{\cX}$ and $\cP(\fG^{\cX})$ respectively, both equipped with the weak topology.
That is, 
a sequence of probability measures 
$(\mu_n)_n$
converges in $\cP(\fG^{\cX})$
iff 
$(\int f\dd\mu_n)_n$ is convergent for any
bounded continuous function $f:\fG^{\cX}\rightarrow \mathbb R$.
It is well known that both spaces are Polish, 
 i.e. separable and completely metrizable \cite{Polish}. 
As a consequence, in order
to prove convergence  of
$(\mu_n)_n$ in $\cP(\fG^{\cX})$ it suffices to show, that the sequence $(\int f\dd\mu_n)_n$ is convergent for any continuous function $f:\fG^{\cX}\rightarrow \mathbb R$ 
with compact support.  The same characterization of weak convergence holds in $\cP^2(\fG^{\cX})$.

For a factor graph $F$ and $x\in \cV(F)\cup\cF(F)$ let $\mu_{\nix}(t|F_x)$ denote the map $y\mapsto\mu_y(t|F_x)$.
Define
	\begin{align*}
	\lambda_{k,F,t}&:=\frac1{|\cV(F)|}\sum_{v\in \cV(F)}\delta_{[F_v,v,\mu_{\nix}(t|F_v)]}\in\cP(\fF^{\cbc{0,1}})
        \end{align*}
and
        \begin{align}\label{eqLambdadknt}
	\Lambda_{d,r,k,n,t}&:=\Erw_{\vF}[\atom_{\lambda_{k,\vF,t}}]\in\cP^2(\fF^{\cbc{0,1}}).
	\end{align}
Thus, $\Lambda_{d,r,k,n,t}$ is the distribution of the neighbourhoods of the random graph $\vF$ marked according to $\mu_{\nix}(t|\vF)$.

\begin{lemma}\label{Prop_WP}
We have \ $\lim_{n\to\infty}\Lambda_{d,r,k,n,t}=\atom_{\cL([\T_{t}(d,r,k)])}$ for any $t\geq0$.
\end{lemma}
\begin{proof}
Let $\tau$ be a locally finite acyclic factor graph rooted at $v_0$ and let $F$ be a factor graph such that $\partial^{s+t}[F,v]\ism\partial^{s+t}[\tau,v_0]$
for some variable node $v\in\cV(F)$. Assume we would like to study Warning Propagation on $\partial^s[F,v]$ after $t$ iterations on $F$. By definition of Warning Propagation,
all Warning Propagation marks of vertices in $\partial^s[F,v]$ are determined by the initialization of Warning Propagation and the structure of $\partial^{s+t}[F,v].$ Therefore, if  
$\partial^{s+t}[F,v]\ism\partial^{s+t}[\tau,v_0]$ holds, we obtain that 
$\partial^s[F,v,\mu_{\nix}(t|F)]\ism\partial^s[\tau,v_0,\mu_{\nix}(t|\tau)].$
The assertion follows from \Thm~\ref{Thm_GW}.
\qquad
\end{proof}

\begin{lemma}\label{lem_aux}
We have \ $\lim_{n\to\infty}\Lambda_{d,r,k,n}=\lim_{t\to\infty}\lim_{n\to\infty}\Lambda_{d,r,k,n,t}$.
\end{lemma}
\begin{proof}
Let $f:\cP(\fG^{\cbc{0,1}})\to\RR$ be continuous 
with compact support.
\Lem~\ref{Prop_WP} implies that $(\Lambda_{d,r,k,n,t})_{n}$ converges for any $t$.
Hence, so does $(\int f\dd\Lambda_{d,r,k,n,t})_{n}$.
Therefore, we will compare $\int f\dd\Lambda_{d,r,k,n}$ and $\int f\dd\Lambda_{d,r,k,n,t}$.
Plugging in the definitions of $\Lambda_{d,r,k,n}$ and $\Lambda_{d,r,k,n,t}$ (\eqref{eqLambdadkn} and~\eqref{eqLambdadknt}), we find that

	\begin{align*}
	\int f\dd\Lambda_{d,r,k,n}&=\Erw_{\vF}\brk{f(\lambda_{k,\vF})},&
	\int f\dd\Lambda_{d,r,k,n,t}&=\Erw_{\vF}\brk{f(\lambda_{k,\vF,t})}.
	\end{align*}
Hence,
	\begin{align}\label{eqconv666}
	\abs{\int f\dd\Lambda_{d,r,k,n}-\int f\dd\Lambda_{d,r,k,n,t}}\leq\Erw_{\vF}\abs{f(\lambda_{k,\vF})-f(\lambda_{k,\vF,t})}.
	\end{align}
To bound the last term we will show that
\begin{align}\label{convp}
 \abs{\int\kappa_{F,s}\dd\lambda_{k,\vF}-\int\kappa_{F,s}\dd\lambda_{k,\vF,t}}\overset{t,n\to\infty}{\longrightarrow}0
\end{align}
in probability for all $F\in\cP(\fF_{\cbc{0,1}})$ and $s\geq 0$.
Plugging in the definitions of $\lambda_{k,\vF}$ and $\lambda_{k,\vF,t}$ we obtain
\begin{align*}
 &\Erw_{\vF}\abs{\int\kappa_{F,s}\dd\lambda_{k,\vF}-\int\kappa_{F,s}\dd\lambda_{k,\vF,t}}&\\
 &\quad\quad\leq\Erw_{\vF}\brk{\frac{1}{n}\sum_{v\in [n]}\abs{\int\kappa_{F,s}\dd\atom_{\brk{\vF_v,v,\sigma_{k,\vF_v}}} - 
 \int\kappa_{F,s}\dd\atom_{\brk{\vF_v,v,\mu_{\nix}(t|\vF_v)}} }}\\
 &\quad\quad=\Erw_{\vF,\vv}\left[\abs{\kappa_{F,s}([\vF_{\vv},\vv,\sigma_{k,\vF_{\vv}}])-\kappa_{F,s}([\vF_{\vv},\vv,\mu_{\nix}(t,\vF_{\vv})])}\right].
\end{align*}
By the definition of $\kappa_{F,s}$ we have
	\begin{align}
	\Erw_{\vF,\vv}\left[\abs{\kappa_{\Gamma,s}([\vF_{\vv},\vv,\sigma_{k,\vF_{\vv}}])-\kappa_{\Gamma,s}([\vF_{\vv},\vv,\mu_{\nix}(t,\vF_{\vv})])}\right]\nonumber\\
        \leq\pr\brk{\partial^s[\vF_{\vv},\vv,\sigma_{k,\vF_{\vv}}]\neq\partial^s[\vF_{\vv},\vv,\mu_{\nix}(t|\vF_{\vv})]}.\label{eqconv667}
	\end{align}
If $v$ is such that $\partial^s[\vF_{\vv},\vv,\sigma_{k,\vF_{\vv}}]\neq\partial^s[\vF_{\vv},\vv,\mu_{\nix}(t|\vF_{\vv})]$, then by \Lem~\ref{Lemma_WP} this means that after $t$ 
iterations of Warning Propagation on $\vF$ there is a vertex $u$ at distance at most $s$ from $v$ such that $\mu_u(t|\vF)=1$, but $u\not\in\cC_k(\vF).$ By \Lem~\ref{Lemma_eps}, we can bound 
the probability that such $u$ exists provided that there are not too many vertices at distance at most $s$ from $v$ in total. By \Thm~\ref{Thm_GW} this is the case for most vertices
$v$. Therefore we can bound the right hand side of (\ref{eqconv667}).

To make this more precise let $\delta>0$ and denote by $\cI(\ell,\vF)$ the set of all variable nodes $v$ in $\vF$ such that $\partial^s[\vF_v,v]$ contains more than $\ell$ variable and factor nodes.
\Thm~\ref{Thm_GW} implies that there exists $\ell_0=\ell_0(\delta,d)$ such that $\pr\brk{|\cI(\ell_0,\vF)|>\delta n}\leq\delta$.
Further, let $\cJ(\vF)$ be the set of all variable nodes $v\in[n]\setminus \cI(\ell_0,\vF)$ such that there is a vertex $u\in\partial^s[\vF_v,v]$ with 
$u\not\in\cC_k(\vF)$, but $\mu_u(t|\vF)=1$.
Then \Lem~\ref{Lemma_eps} implies that $\pr\brk{|\cJ(\vF)|>\delta n}\leq\delta$, provided that $n\geq n_0(\delta)$ and $t\geq t_0(\delta)$ are sufficiently large.
Hence, conditioning on the event $\{|\cI(\ell_0,\vF)|\leq\delta n,|\cJ(\vF)|\leq\delta n\}$ and its complement implies 
	\begin{align*}
	&\pr\brk{\partial^s[\vF_{\vv},\vv,\sigma_{k,\vF_{\vv}}]\neq\partial^s[\vF_{\vv},\vv,\mu_{\nix}(t,\vF_{\vv})]}&\\
        &\quad\quad\leq\pr\brk{|\cI(\ell_0,\vF)|>\delta n}+\pr\brk{|\cJ(\vF)|>\delta n}\\
        &\quad\quad\quad\quad+\pr\brk{\vv\in\cI(\ell_0,\vF)\cup\cJ(\vF)||\cI(\ell_0,\vF)|\leq\delta n,|\cJ(\vF)|\leq\delta n}\leq4\delta
	\end{align*}
and we obtain (\ref{convp}). Now, let $\eps>0$. 
Using 
that $\cP(\fG_{\cbc{0,1}})$ is metrizable
 and the fact that the functions $\kappa_{F,s}$ generate the topology on $\fF_{\cbc{0,1}}$, (\ref{convp}) implies that for given $\delta=\delta(\eps)>0$ 
there exist $n>n_0(\delta)$, $t>t_0(\delta)$ such that the distance of $\lambda_{\vF,k}$ and 
$\lambda_{\vF,k,t}$ in an appropriate metric is less than $\delta$ with probability larger than $1-\eps$. $f$ is continuous with compact support and thus uniformly continuous. This implies
	\begin{align}\label{eqconv669}
	\Erw_{\vF}\abs{f(\lambda_{k,\vF})-f(\lambda_{k,\vF,t})}<\eps.
	\end{align}
for suitable $\delta>0.$ 

Combining \eqref{eqconv666} and \eqref{eqconv669} we obtain
	$$\lim_{n\to\infty}\Lambda_{d,r,k,n}=\lim_{t\to\infty}\lim_{n\to\infty}\Lambda_{d,r,k,n,t}$$
as desired.
\qquad 
\end{proof}

{\em Proof of \Prop~\ref{Prop_main}.}\\
We begin with proving that the sequence $(\cL([\T_t(d,r,k)]))_{t}$ converges in $\cP(\fG_{\{0,1\}})$.
Since $\cP(\fG_{\{0,1\}})$ is equipped
with the weak topology, we need to show that for any bounded continuous function $f:\fF_{\cbc{0,1}}\to\RR$
the sequence 
	$(\Erw_{\T(d,r)}[f([\T_t(d,r,k)])])_t$ converges.
Since the topology on $\fF_{\cbc{0,1}}$ is generated by the functions $\kappa_{F,s}$ it suffices to 
assume that $f=\kappa_{F,s}$ for some $F\in\fF_{\cbc{0,1}}$, $s\geq0$.
Hence,
	\begin{align*}
	\Erw_{\T(d,r)}[f([\T_t(d,r,k)])]&=\pr\brk{\partial^s[\T_t(d,r,k)]=\partial^sF}.
	\end{align*}
To show that $(\pr\brk{\partial^s[\T_t(d,r,k)]=\partial^sF})_t$ converges, let $\eps>0$, $q\geq 0$ and let $\cT_{q}$ be the $\sigma$-algebra generated
by $\partial^{q}[\T(d,r)]$. Let $\cB(q)$ be the set of vertices at level $s$ in $\T(d,r)$. We will distinguish between the cases that $s$ is odd and even.

First, let $s$ be odd. Then $\cB(s+1)$ consists of variable nodes.
The structure of the random tree $\T(d,r)$ is recursive, i.e., given $\cT_{s+1}$ the tree pending on each vertex $v\in\cB(s+1)$ is
just an independent copy of $\T(d,r)$ itself.
Therefore, \Lem~\ref{Lemma_as} and the union bound imply that conditioned on $\cT_{s+1}$ 
the limits
	$$\mu_{v\toboss}(\T(d,r))=\lim_{t\to\infty}\mu_{v\toboss}(t|\T(d,r))\qquad(v\in\cB(s+1))$$
exist almost surely.

All limits $\lim_{t\to\infty}\mu_{x}(t|\T(d,r))$ for
nodes $x$ at distance at most $s$ from $v_0$ are determined by the limits $\mu_{v\toboss}(\T(d,r))$ for $v\in\cB(s+1)$.
This implies that the sequence $(\pr\brk{\partial^s[\T_t(d,r,k)]=\partial^sF|\cT_{s+1}})_t$ of random variables converges almost surely.
This immediately implies that
$(\pr\brk{\partial^s[\T_t(d,r,k)]=\partial^s F})_t$ converges almost surely.
If $s$ is even, we can use the same argument for variable nodes in $\cB(s+2)$ and conditioning on $\cT_{s+2}$ to show that the sequence
$(\pr\brk{\partial^s[\T_t(d,r,k)]=\partial^sF|\cT_{s+2}})_t$ converges. Again, this implies convergence of the sequence  
$(\pr\brk{\partial^s[\T_t(d,r,k)]=\partial^s F})_t$.
As this holds for any $\Gamma,s$, the limit $\theta_{d,r,k}=\lim_{t\to\infty}\cL([\T_t(d,r,k)])$ exists.
Therefore, it holds that
$\lim_{t\to\infty}\atom_{\cL([\T_{t}(d,r,k)])}=\atom_{\theta_{d,r,k}}$.
Finally, by 
 \Lem~\ref{lem_aux} and \Lem~\ref{Prop_WP} we obtain
$\lim_{n\to\infty}\Lambda_{d,r,k,n}=\atom_{\theta_{d,r,k}},$
which completes the proof.
\qquad
\endproof

\section{The multi-type branching process}
In this section we 
transfer the analysis from
\cite{kcore} to the factor graph notion 
to
prove that the $\{0,1\}$-marked tree $\T_t(d,r,k)$ converges to the projection $\T(d,r,k,p^*)$ of the $9$-type
branching process $\hat\T(d,r,k,p^*)$ from \S~\ref{sec_results}
as $t\to\infty.$ 
This will imply \Thm~\ref{Thm_lwc}.
\subsection{Truncating the tree}
We begin with characterizing the limiting distribution of the first $s$ levels of $\T_t(d,r,k)$. Though our ultimate interest is in the Warning Propagation marks 
on $\T(d,r)$, a first important step of our proof is studying the bottom up messages $\mu_{x\toboss}(t|\T(d,r))$ at nodes $x$ in $\cV(\T(d,r))\cup\cF(\T(d,r))$. By construction of the
Warning Propagation marks, starting at the root $v_0$ these can be reconstructed from the bottom up messages. That is, the bottom up messages contain all the necessary information to
define the Warning Propagation marks on $\T(d,r).$ The key feature of the bottom up messages $\mu_{x\toboss}(t|\T(d,r))$ is that they
are independent of the part of $\T(d,r)$ above $x$, i.e. they are determined by the subtree of $\T(d,r)$
pending on $x$. By the recursive structure of $\T(d,r)$ conditioned on the structure of $\T(d,r)$ above a variable node $v$ the bottom up message at $v$ has the same distribution as the message 
$\mu_{v_0\toboss}(t|\T(d,r))$ at the root $v_0$. This behaviour makes it much easier to first describe the limiting distribution of
$\T(d,r)$ marked with these bottom up messages $\mu_{x\toboss}(t|\T(d,r))$, i.e.
$$\theta_{d,r,k,t}^{s}:=\cL(\partial^s[\T(d,r),v_0,\mu_{\nix\toboss}(t|\T(d,r))]).$$

Using \Lem~\ref{Lemma_as} and the recursive structure of $\T(d,r)$ we can conjecture a possible limit distribution 
for $\theta_{d,r,k,t}^{s}$ as $t\to\infty$. \Lem~\ref{Lemma_as} suggests to assume that the pointwise limits of the bottom up messages $\mu_{v\toboss}(t|\T(d,r))$
at variable nodes $v\in\cV(\T(d,r))$ exist. We first apply this to the case that $s$ is even.
Once  we condition on the isomorphism class $\partial^s[\T(d,r)]$ of the tree up to level $s$, the messages $\lim_{t\to\infty}\mu_{x\toboss}(t|\T(d,r))$ for nodes
$x$ at distance less than $s$ from the root $v_0$ are determined by the boundary messages 
$\lim_{t\to\infty}\mu_{v\toboss}(t|\T(d,r))$ sent out by the variable nodes at distance precisely $s$ from $v_0$.
Furthermore, the limit $\lim_{t\to\infty}\mu_{v\toboss}(t|\T(d,r))$ is determined by the tree pending on $v$ only.
These trees are mutually independent copies of $\T(d,r)$.
Thus, \Lem~\ref{Lemma_as} suggests that the ''boundary messages`` converge to a sequence of mutually independent $\Be(p^*)$ variables.
Therefore, we can conjecture that $\theta_{d,r,k,t}^{s}$ converges to the random $\{0,1\}$-marked tree which is obtained as follows. Create the first $s$ levels
of the random tree $\T(d,r)$ (without markings), then create a sequence of independent $\Be(p^*)$ random variables. Plug in these ''messages`` at level $s$ and pass them 
up to the root.

If $s$ is odd the nodes at distance precisely $s$ from the root $v_0$ are factor nodes and thus \Lem~\ref{Lemma_as} does not apply to the messages at these nodes. Still,
conditioning on the isomorphism class $\partial^{s+1}[\T(d,r)]$ we can apply the same argument for the variable nodes at distance $s+1$ from the root. In this case, one would create the first $s+1$ 
levels of $\T(d,r)$ marking each variable node at distance $s+1$ by a $\Be(p^*)$ ''message`` and pass the messages up to the root.

To define this distribution formally, let
$T$ be a locally finite acyclic factor graph rooted at $v_0\in\cV(T)$ and let $s>0$ be an integer.
Moreover, let $\beta=(\beta_w)_{w\in \cV(T)}$ be a family of independent $\Be(p^*)$ random variables.
If either $t=0$ or if $v$ is a variable node at distance greater than $s$ from $v_0$, we define
	\begin{equation}\label{eqIniv}
	\mu_{v\toboss}^*(t|T,s)=\beta_v.
	\end{equation}
If $a$ is a factor node in $T$
        \begin{equation}\label{eqInie}
	\mu_{a\toboss}^*(0|T,s)=1.
	\end{equation}
Moreover, if $t>0$ and if $a$ is a factor node in $T$
	\begin{equation}\label{eqIt1e}
	\mu_{a\toboss}^*(t|T,s)=\vecone\cbc{\sum_{v\in\partial^+ v}\mu_{v\toboss}^*(t-1|T,s)=r-1}.
	\end{equation}
Finally, if $t>0$ and if $v$ has distance less than or equal to $s$ from $v_0$, let
	\begin{equation}\label{eqIt1v}
	\mu_{v\toboss}^*(t|T,s)=\vecone\cbc{\sum_{a\in\partial^+v}\mu_{a\toboss}^*(t-1|T,s)\geq k-1}.
	\end{equation}
Let us denote the map $x\mapsto\mu_{x\toboss}^*(t|T,s)$ by $\mu_{\nix\toboss}^*(t|T,s)$.

\begin{lemma}\label{claim_5}
Let $0<s<u$. We have $\cL(\partial^s[\T(d,r),v_0,\mu_{\nix\toboss}^*(s|\T(d,r),s)])
\linebreak[4]
=\cL(\partial^s[\T(d,r),v_0,\mu_{\nix\toboss}^*(u|\T(d,r),s)])$.
\end{lemma}
\begin{proof}
By construction, the messages $\mu^*_{x\toboss}(s|\T(d,r),s)$ in $\T(d,r)$ remain fixed after $s$ iterations of the update rules \eqref{eqIniv} -- \eqref{eqIt1v}.
\qquad
\end{proof}

\noindent
The main result of this subsection is the following.

\begin{lemma}\label{Lemma_sizeless}
We have $\lim_{t\to\infty}\theta_{d,r,k,t}^{s}=\theta_{d,r,k}^{s,*}$ for all $s\geq 0$.
\end{lemma}
\begin{proof}
We will distinguish between the cases that $s$ is even and odd. If $s$ is even, let $0<\eps<1/10$.
We couple the distributions of $\partial^s[\T(d,r),v_0,\mu_{\nix\toboss}(t|\T(d,r))]$ and $\partial^s[\T(d,r),v_0,\mu_{\nix\toboss}^*(s|\T(d,r),s)]$ such that both
operate on the same tree $\T(d,r)$. To this end, let $\cB=\cB(s)$ be the set of all variable nodes of the random tree $\T(d,r)$ that have distance precisely $s$ from $v_0$
and let
$\cE=\cbc{\mu_{v\toboss}(t|\T(d,r))=\beta_v \forall t>t_1,v\in\cB}$.
If the event $\cE$ occurs, then the initialization $\mu_{v\toboss}^*(t|\T(d,r),s)=\beta_v$ for $v\in\cB$ (cf.\ (\ref{eqIniv})),
and (\ref{eqIt1v}) ensure that
	$$\mu_{u\toboss}^*(s|\T(d,r),s)=\mu_{u\toboss}^*(t|\T(d,r),s)=\mu_{u\toboss}(t|\T(d,r))\qquad\mbox{for all $t>t_1+s$}.$$
On the other hand it holds that $\Erw_{\T(d,r)}[|\cB|]$ is bounded. This implies that 
there exists $C=C(d,r,k,\eps)>0$ such that 
	\begin{equation}\label{eqsizeless1}
	\pr\brk{|\cB|\leq C}>1-\eps/2.
	\end{equation}
Now, 
let $\beta=(\beta_v)_{v\in\cB}$ be a family of mutually independent $\Be(p^*)$ random variables.
Given the sub-tree $\partial^s[\T(d,r),v_0]$, the trees $\T_v$ pending on the variable nodes $v\in\cB$
are mutually independent and have the same distribution as the tree $\T(d,r)$ itself.
Therefore, \Lem~\ref{Lemma_as} implies that given $|\cB|\leq C$ there exist $t_1=t_1(d,r,k,\eps)$ 
and a coupling of $\beta$ with the trees $(\T_v)_{v\in\cB}$ such that
	\begin{equation}\label{eqsizeless2}
	\pr\brk{\mu_{v\toboss}(t|\T(d,r))=\beta_v	\forall t>t_1,v\in\cB	\; {\big |} \;|\cB|\leq C}>1-\eps/2.
	\end{equation}
Hence, (\ref{eqsizeless1}) and (\ref{eqsizeless2}) yield
	$$\pr\brk{\partial^s[\T(d,r),v_0,\mu_{\nix\toboss}(t|\T(d,r))]=\partial^s[\T(d,r),v_0,\mu_{\nix\toboss}^*(s|\T(d,r),s)]}>1-\eps,$$
as desired. If $s$ is odd we apply the same argument to the set $\cB=\cB(s+1)$ of variable nodes at distance precisely $s+1$ from the root. 
\qquad
\end{proof}

\subsection{Turning the tables}
The aim of this section is to introduce a top-down process that matches the distribution $\theta_{d,r,k}^*$.
To this end, let $\T^*(d,r,k)$ be the random $\cbc{0,1}$-marked acyclic factor graph defined as follows.
Recall that $c=c(d,r)=d/(r-1)!$.
Initially, $\T^*(d,r,k)$ consists of a root $v_0$ which is marked with $1$ with probability $p^*$ and $0$ otherwise. 
Starting at $v_0$, every variable node of type $0$ has a $\Po\br{c(1-{p^*}^{r-1})}$ number of factor node children of type $0$
and an independent $\Po_{<k-1}\br{c{p^*}^{r-1}}$ number of factor node children of type $1$.
Further, a type $1$ variable node spawns $\Po\br{c(1-{p^*}^{r-1})}$ type $0$ offspring and independently 
$\Po_{\geq k-1}\br{c{p^*}^{r-1}}$ type $1$ offspring.
Each factor node has precisely $r-1$ children. If a factor node is of type $1$, then all 
its children are of type $1$ as well. If a factor node is of type $0$, then it has a $\Bin_{<r-1}(r-1,p^*)$ number of function node children of type $1$. 
The remaining ones are marked with $0$.
The mark of each vertex $v$, denote by $\mu_{v\toboss}^*$, is identical to its type.

\begin{lemma}\label{Lemma_topdown}
For any $s\geq0$ we have $\cL\bc{\partial^s[\T^*(d,r,k)]}=\theta_{d,r,k}^{s,*}$.
\end{lemma}

\begin{proof}
Let
	$$\cT(s,u)=\partial^{s}[\T(d,r),v_0,\mu_{\nix\toboss}^*(u|\T(d,r),s)],\qquad\cT(s)=\partial^s[\T^*(d,r,k)]$$
Then our aim is to prove that $\cL(\cT(s,s))=\cL(\cT(s))$ for all $s$.
The proof is by induction on $s$.
In the case $s=0$ both $\cT(s)$ and $\cT(s,s)$ consist of the root $v_0$ only,
which is marked $1$ with probability $p^*$ and $0$ otherwise.

Now, assume that $\cL(\cT(s))=\cL(\cT(s,s))$ holds for even $s$. 
Let $\beta_v=\mu_{v\toboss}^*(0|\T(d),s+2)$ for $v\in\cB(s+2)$ be independent $\Be(p^*)$ variables.
Further, let $X_v(z)$ be the number of children $b$ of $v\in\cB(s)$ such that $\mu_{b\toboss}^*(1|\T(d),s+1)=z$.
Clearly, in the random tree $\T(d,r)$ the total number of children of $v\in\cB(s)$ has distribution $\Po(c)$, and these numbers are mutually independent
	conditioned on $\partial^s\T(d,r)$.
Since for each child $b$ we have $\mu_{b\toboss}^*(1|\T(d),s+1)=1$ with probability ${p^*}^{r-1}$ independently, we obtain that 
$X_v(z)$ has distribution $\Po(c{p^*}^{r-1})$ if $z=1$ and distribution $\Po(c(1-{p^*}^{r-1}))$ if $z=0$.
Further, $\mu_{v\toboss}^*(s+1|\T(d,r),s+1)=1$ iff $X_v(1)\geq k-1$.

Hence, the distribution of $\cT(s+1,s+1)$ conditioned on $\cT(s,s+1)$ can be described as follows.
Conditioned on $\cT(s,s+1)$, $(X_u(z))_{u\in\cB(s),z\in\cbc{0,1}}$ are independent random variables.
Furthermore, conditioned on $\mu_{v\toboss}^*(s+1|\T(d),s+1)=1$, $X_v(1)$ has distribution $\Po_{\geq k-1}(c{p^*}^{r-1})$.
Given $\mu_{v\toboss}^*(s+1|\T(d,r),s+1)=0$, $X_v(1)$ has distribution $\Po_{<k-1}(c{p^*}^{r-1})$.
In addition, $X_v(0)$ has distribution $\Po(c(1-{p^*}^{r-1}))$ for any $v$.

Now, let $s$ be odd, i.e. the set $\cB(s+1)$ consists of variable nodes. In this case the distribution $\cL(\cT(s+1,s+1))$ can be described as follows. 
Again, create $\T(d,r),$ further let $\beta_v=\mu_{v\toboss}^*(0|\T(d,r),s+1)$ for $v\in\cB(s+1)$ be the boundary messages at level $s+1.$
Then $(\beta_v)_{v\in\cB(s+1)}$ is a family of independent $\Be(p^*)$ variables. Again, let 
$X_a(z)$ be the number of children $v$ of $a\in\cB(s)$ such that $\beta_v=z$. For each of the $r-1$ children $v$ of $a$ we have $\beta_v=1$ with probability $p^*$ independently.
We see that conditioned on $\partial^s\T(d,r)$ the random variables
	$(X_a(z))_{a\in\cB(s),z\in\cbc{0,1}}$ are mutually independent.
Moreover, $X_a(1)$ has distribution $\Bin(r-1,p^*)$ and $X_a(0)=r-1 - X_a(1)$.
Further, $\mu_{a\toboss}^*(s+1|\T(d),s+1)=1$ iff $X_a(1)= r-1$.

Therefore,
given $\mu_{a\toboss}^*(s+1|\T(d),s+1)=1$, $X_a(1)=r-1$ and $X_a(0)=0$.
In addition, given $\mu_{a\toboss}^*(s+1|\T(d),s+1)=0$, $X_a(1)$ has distribution $\Bin_{< r-1}(r-1,p^*)$ and $X_a(0)=r-1-X_a(1)$ for any $a$.

Therefore, the distribution of the random variables $(X_x(z))_{x\in\cB(s),z\in\cbc{0,1}}$ conditioned on $\cT(s,s+1)$ coincides
with the offspring distribution of the tree $\T^*(d,r,k)$ for all $s\geq 0$. Since $\cL(\cT(s,s+1))=\cL(\cT(s,s))=\cL(\cT(s))$ by \Lem~\ref{claim_5} and induction, the assertion follows.
\qquad
\end{proof}

\subsection{Exchanging messages both ways}
\Lem s~\ref{Lemma_sizeless} and~\ref{Lemma_topdown} show that the labels $\mu_{x\toboss}^*$ of $\T^*(d,r,k)$ correspond to the
 ''upward messages`` that are sent toward the root in the tree $\T(d,r)$.
Of course, in the tree $\T(d,r)$ the marks $\mu_x(t|\T(d,r))$ 
can be computed from the messages $\mu_{x\toboss}(t|\T(d,r))$.
For the root we have 
	\begin{align*}
	\mu_{v_0}(t|\T(d))&=\vecone\cbc{\sum_{b\in\partial v_0}\mu_{b\toboss}(t|\T(d))\geq k}&\mbox{for all }t\geq0.
	\end{align*}
That is, the Warning Propagation mark at the root does only depend on the messages that $v_0$ receives from its children in $\T(d,r).$ This is certainly not the 
case for the remaining Warning Propagation marks in $\T(d,r),$ since each node beside the root receives an additional message from 
its parent in the tree. These 
``top-down"-messages can recursively
be determined from the messages $\mu_{x\toboss}(t|\T(d,r))$ as follows. For the root we simply have 
	\begin{align}\label{eqTopDown1}
	\mu_{\fromboss v_0}(t|\T(d,r))&=0&\mbox{for all }t\geq0.
	\end{align}
For every node $x$ we initialize $\mu_{x\fromboss}(0|\T(d,r))=1.$
Now, assume that $u$ is the parent of some factor node $a$ in $\T(d,r).$ Then we let $\mu_{\fromboss a}(t+1|\T(d,r))=\mu_{u\to a}(t+1|\T(d,r))$, i.e.
\begin{align}\label{eqTopDown2e}
	\mu_{\fromboss a}(t+1|\T(d,r))=\vecone\cbc{\mu_{\fromboss u}(t|\T(d,r))+\sum_{b\in\partial^+ u\setminus a}\mu_{b\toboss}(t|\T(d,r))\geq k-1}
	\end{align}
for $t\geq 0$. If now $b$ is the parent of some variable node $v$, similarly we let $\mu_{\fromboss v}(t+1|\T(d,r))=\mu_{b\to v}(t+1|\T(d,r))$. That is
	\begin{align}\label{eqTopDown2v}
	\mu_{\fromboss v}(t+1|\T(d,r))=\vecone\cbc{\mu_{\fromboss b}(t|\T(d,r))+\sum_{w\in\partial^+ b\setminus v}\mu_{w\toboss}(t|\T(d,r))=r-1}
	\end{align}
for $t\geq 0$. Finally,
	\begin{align*}
	\mu_{a}(t|\T(d))&=\vecone\cbc{\mu_{\fromboss a}(t|\T(d,r))+\sum_{w\in\partial^+a}\mu_{w\toboss}(t|\T(d,r))=r-1}
	\end{align*}
and 
	\begin{align*}
	\mu_{v}(t|\T(d,r))&=\vecone\cbc{\mu_{\fromboss v}(t|\T(d,r))+\sum_{b\in\partial^+v}\mu_{b\toboss}(t|\T(d,r))\geq k}
	\end{align*}
for $t\geq 0$.

As a next step we will mimic this construction to introduce ''top-down`` messages and marks 
on the random tree $\T^*(d,r,k)$.
To this end, we set $\mu_{\fromboss v_0}^*=0$
and
 $$\mu_{v_0}^*=\vecone\cbc{\sum_{a\in\partial v_0}\mu_{a\toboss}^*\geq k}.$$
Further, assume that $\mu_{\fromboss u}^*$ has been defined already and that $u$ is the parent of some factor node $a$.
Then we let
	\begin{align}\label{eqKathrin1e}
	\mu_{\fromboss a}^*&=\vecone\cbc{\mu_{\fromboss u}^*+\sum_{b\in\partial^+ u\setminus a}\mu_{b\toboss}^*\geq k-1}
	\end{align}
and
	\begin{align}\label{eqKathrin2e}
	\mu_a^*&=\vecone\cbc{\mu_{\fromboss a}^*+\sum_{w\in\partial^+ a}\mu_{w\toboss}^* = r}
	\end{align}
Similarly, assume that $\mu_{\fromboss a}^*$ has been defined already and that $a$ is the parent of a variable node $v$. Then we let
	\begin{align}\label{eqKathrin1v}
	\mu_{\fromboss v}^*&=\vecone\cbc{\mu_{\fromboss a}^*+\sum_{w\in\partial^+ a\setminus v}\mu_{w\toboss}^*= r-1}
	\end{align}
and
	\begin{align}\label{eqKathrin2v}
	\mu_v^*&=\vecone\cbc{\mu_{\fromboss v}^*+\sum_{b\in\partial^+ v}\mu_{b\toboss}^*\geq k }.
	\end{align}
Let $\hat\T^*(d,r,k)$ signify the resulting tree in which each node $x$ is marked by the triple 
	$(\mu_x^*,\mu_{x\toboss}^*,\mu_{\fromboss x}^*)$.
We observe that by the above definitions it holds that
	$$(\mu_v^*,\mu_{v\toboss}^*,\mu_{\fromboss v}^*)\in\{(0,0,0),(0,0,1),(0,1,0),(1,1,0),(1,1,1)\}$$
for all variable nodes $v$ and
	$$(\mu_a^*,\mu_{a\toboss}^*,\mu_{\fromboss a}^*)\in\{(0,0,0),(0,0,1),(0,1,0),(1,1,1)\}$$
for all factor nodes $a$.

\begin{lemma}\label{Lemma_decorated}
For any $s>0$ we have $\lim_{t\to\infty}\cL(\partial^s\hat\T_t(d,r,k))=\cL(\partial^s\hat\T^*(d,r,k))$.
\end{lemma}
\begin{proof}
This is immediate from \Lem~\ref{Lemma_sizeless}, \Lem~\ref{Lemma_topdown} and the fact that
the definitions (\ref{eqTopDown1})--(\ref{eqTopDown2v}) and (\ref{eqKathrin1e}) and (\ref{eqKathrin1v}) 
of the ''top-down`` messages for $\hat\T_t(d,r,k)$ and $\hat\T^*(d,r,k)$ match.
\qquad
\end{proof}

\subsection{Proof of the main theorem}\label{sec_mainproof}
Finally, we show that the tree $\hat\T^*(d,r,k)$ has the same distribution as the $9$-type branching process 
from \S~\ref{sec_results}.

\begin{lemma}\label{Lem_id}
We have $\cL(\hat\T^*(d,r,k))=\cL(\T(d,r,k,p^*))$.
\end{lemma}
\begin{proof}
We show that $\cL(\partial^s[\hat\T^*(d,r,k)])=\cL(\partial^s[\hat\T(d,r,k,p^*)])$ for any $s\geq0$ by induction on $s$. 
Let $\cF_s$ be the $\sigma$-algebra generated by the $\cbc{0,1}$-marked tree $\partial^s \T^*(d,r,k)$.
In addition, let $\hat\cF_s$ be the $\sigma$-algebra generated by $\partial^s\hat\T^*(d,r,k)$.
Given $\partial^s \T^*(d,r,k)$ including the marks, by construction one can construct the missing bits of the types of vertices in $\partial^s\hat\T^*(d,r,k)$.
On the other hand, given $\partial^s\hat\T^*(d,r,k)$ we can reconstruct the the marks at level $s+1$ in $\partial^{s+1} \T^*(d,r,k)$.
Therefore, it holds that
	\begin{equation}\label{Fact_measurable}
	\cF_s\subset\hat\cF_s\subset\cF_{s+1}\qquad\mbox{for any $s\geq0$.}
	\end{equation}

With respect to $s=0$, we see that $\mu_{\fromboss v_0}^*=0$ with certainty.
Moreover, $\mu_{v_0\toboss}^*$ has distribution $\Be(p^*)$ and $\mu_{v_0}^*=0$ if $\mu_{v_0\toboss}^*=0$.
On the other hand, conditioned on that $\mu_{v_0\toboss}^*=1$, $v_0$ has at least $k-1$ children $a$ such that $\mu_{a\toboss}^*=1$.
Therefore, $\sum_{a\in\partial v_0}\mu_{a\toboss}^*$ 
has distribution $\Po_{\geq k-1}(c{p^*}^{r-1})$, and $\mu_{v_0}^*=1$ iff $\sum_{a\in\partial v_0}\mu_{a\toboss}^*\geq k$.
Hence, using the fixed point property $p^*=\pr\br{\Po(c{p^*}^{r-1})\geq k-1}$ and~\eqref{type_prob} with $q=q(d,k,p^*)$, we obtain
\begin{align*}
 \pr\br{(\mu_{v_0}^*,\mu_{v_0\toboss}^*,\mu_{\fromboss v_0}^*)=000}&=\pr\br{\Po(c{p^*}^{r-1})<k-1}=1-p^*=p_{000},\\
 \pr\br{(\mu_{v_0}^*,\mu_{v_0\toboss}^*,\mu_{\fromboss v_0}^*)=010}&=\pr\br{\Po(c{p^*}^{r-1})=k-1}=p^*q=p_{010},\\
\pr\br{(\mu_{v_0}^*,\mu_{v_0\toboss}^*,\mu_{\fromboss v_0}^*)=110}&=\pr\br{\Po(c{p^*}^{r-1})\geq k}=p^*(1-q)=p_{110}.
\end{align*}

To proceed from $s$ to $s+1$ we will determine the distribution of $\partial^{s+2}[\T^*(d,r,k)]$ conditioned on $\hat\cF_{s}$.
This is sufficient to determine the distribution of $\T^*(d,r,k)$ given $\hat\cF_{s+1}$.
To this end, let $\cB_s$ be the set of all nodes $x$ at level $s$ in $\T^*(d,r,k)$.
Moreover, for each $x\in\cB_s$ let $\tau_x=(\tau_x(z_1,z_2,z_3))_{z_1,z_2,z_3\in\cbc{0,1}}$ 
be the number of children of $x$ marked $z_1z_2z_3$.
In addition, let
$\tau_x(z_2)$ be the number of children that send a messages of type $z_2$ to $x$.

We will show that the the distribution of $\tau_x(z_1,z_2,z_3)$ coincides with the offspring distributions defined in Figures~\ref{Fig_g} and~\ref{Fig_f}. To begin with, we 
observe that conditioned on $\hat\cF_s$ the random variables $(\tau_x)_{x\in\cB_s}$ are mutually independent. Furthermore, the distribution of $\tau_x$ is determined by the type 
$(\mu_x^*,\mu_{x\toboss}^*,\mu_{\fromboss x}^*)$ of $x$. To prove that the distribution of $\tau_x$ is given by the generating functions we consider each
possible case one after the other.
Recall that factor nodes in $\hat\T(d,r,k,p^*)$ and $\hat\T^*(d,r,k)$ are of possible types $\edgel$
and variable nodes in both trees are of possible types $\vertexl$.

We begin with the case that $s$ is odd, i.e. $\cB(s)$ is a set of factor nodes.
\begin{description}
\item[Case 1:  $(\mu_a^*,\mu_{ a\toboss}^*,\mu_{\fromboss a}^*)=(0,0,0)$:] 
	Since $\mu_{a\fromboss}^*=0$ by (\ref{eqKathrin1v}) we have $\mu_{\fromboss v}^*=0$ for all children $v$ of $a$.
	Further, since $\mu_{ a\toboss}^*=0$, we know that $\tau_a(1)<r-1$.
	Thus, $\tau_a(1)$ has distribution $\Bin_{<r-1}(r-1,p^*)$. Since $a$ has exactly $r-1$ children we have  
        $\tau_a(0)=\tau_a(0,0,0)=r-1 -\tau_a(1)$.
	Further, given that $v$ is a child $a$ such that $\mu_{v\toboss}^*=1$, we have $\mu_{v}^*=1$ iff
		$v$ has at least $k$ children $b$ such that $\mu_{b\toboss}^*=1$ (because $\mu_{v\fromboss}^*=0$).
	This event occurs with probability $1-q$ independently for each $v$.
	Hence, given $\tau_a(1)$ we have $\tau_a(0,1,0)=\Bin(\tau_a(1),q)$ and $\tau_a(1,1,0)=\tau_a(1)-\tau(0,1,0)$.
	In summary, we obtain the generating function $f_{000}$.
\item[Case 2: $(\mu_a^*,\mu_{ a\toboss}^*,\mu_{\fromboss a}^*)=(0,0,1)$:] 
	there are two sub-cases.
	\begin{description}
	\item[Case 2a: $\tau_a(1)=r-2$:]
		Then for any child $w$ of $a$ we have $\mu_{\fromboss w}^*=1-\mu_{w\toboss}^*$.
                Therefore $a$ has exactly one child of type $(0,0,1)$ and $r-2$ children of type $(1,1,0)$ or $(0,1,0)$.
		Hence, for each of the $r-2$ children $w$ such that $\mu_{w\toboss}^*=1$ we have $\mu_w^*=1$ 
		iff $w$ has at least $k$ children $b$ such that $\mu_{b\toboss}^*=1$.
		Thus, $\mu_w^*=\Be(1-q)$ independently for all each such $w$.
	\item[Case 2b: $\tau_a(1)<r-2$:]
		We have $\mu_{\fromboss w}^*=0$ for all children $w$ of $a$.
		Hence, $\tau_a(0)=\tau_a(0,0,0)$.
		Further, $\tau_v(1)$ has distribution $\Bin_{<r-2}(r-1,p^*)$ and $\tau_a(0)=r-1-\tau_a(1)$.
		Finally, since $\mu_{\fromboss w}^*=0$ for all $w$, any child $w$ such that $\mu_{w\toboss}^*=1$ satisfies $\mu_{w}^*=1$ iff
		$w$ has at least $k$ children $b$ such that $\mu_{b\toboss}^*=1$.
		This event occurs with probability $1-q$ independently for all such $w$.
	\end{description}
	Since the first sub-case occurs with probability $\tilde q$ and the second one accordingly with probability $1-\tilde q$, we obtain
	the generating function $f_{001}$.
\item[Case 3: $(\mu_a^*,\mu_{ a\toboss}^*,\mu_{\fromboss a}^*)=(0,1,0)$:] 
	because $\mu_{a\toboss }^*=1$, we have $\tau_a(1)=r-1$ with certainty.
        Using $\mu_{a\fromboss}=0$, we obtain $\mu_{w\fromboss}^*=0$ for all children $w$ of $a$. Therefore all children are of type 
        $(0,1,0)$ or $(1,1,0)$ and 
        if $w$ is a child such that $\mu_{w\toboss}^*=1$, then $\mu_w^*=1$ iff $w$ has at least $k$ children $b$ such that $\mu_{b\toboss}^*=1$.
	This event occurs with probability $1-q$ for each $w$ independently.
	Consequently, $\tau_v(0)=0$ and $\tau_v(1)=\tau_v(1,1,0)+\tau_v(0,1,0)=r-1$ and $\tau_v(1,1,0)=\Bin(r-1,1-q)$.
	Thus, the offspring distribution of $a$ is given by $f_{010}$.
\item[Case 4: $(\mu_a^*,\mu_{ a\toboss}^*,\mu_{\fromboss a}^*)=(1,1,1)$:] 
	Since $\mu_{a\toboss}^*=1$ we know that $\tau_a(1)=r-1$. Since $\mu_a^*=1$, it holds that $\mu_{w\fromboss}^*=1$ for all children $w$ of $a$. 
	In summary, $a$ has exactly $r-1$ children of type $(1,1,1)$ and the distribution of the offspring of $a$ is given by $f_{111}$.
\end{description}
We proceed with the same analysis for even $s$, i.e.
$\cB(s)$ is a set of variable nodes.
\begin{description}
\item[Case 1: $(\mu_v^*,\mu_{ v\toboss}^*,\mu_{\fromboss v}^*)=(0,0,0)$:] 
	Using $,\mu_{ v\toboss}^*=\mu_{\fromboss v}^*=0$ and (\ref{eqKathrin1e}) we have $\mu_{\fromboss b}^*=0$ for all children $b$ of $v$.
	Further, since $\mu_{ v\toboss}^*=0$, we know that $\tau_v(1)<k-1$.
	Thus, $\tau_v(1)=\tau_v(0,1,0)$ has distribution $\Po_{<k-1}\br{c{p^*}^{r-1}}$ independent of $\tau_v(0)$.
        Furthermore, $\tau_v(0)=\tau_v(0,0,0)$ has distribution $\Po\br{c\br{1-{p^*}^{r-1}}}$.
	In summary, we obtain the generating function $g_{000}$. 
\item[Case 2: $(\mu_v^*,\mu_{ v\toboss}^*,\mu_{\fromboss v}^*)=(0,0,1)$:] 
	there are two sub-cases.
	\begin{description}
	\item[Case 2a: $\tau_v(1)=k-2$:]
		Then for any child $b$ of $v$ we have $\mu_{\fromboss b}^*=1-\mu_{b\toboss}^*$.
		With $\tau_v(1)=\tau_v(0,1,0)$ we obtain, that $v$ has $k-2$ children of this type.
		Furthermore, $\tau_v(0)=\tau_v(0,0,1)$ is an independent random variable with distribution $\Po\br{c\br{1-{p^*}^{r-1}}}$.
	\item[Case 2b: $\tau_v(1)<k-2$:]
		We have $\mu_{\fromboss b}^*=0$ for all children $b$ of $v$.
		Hence, $\tau_v(0)=\tau_v(0,0,0)$ has distribution $\Po\br{c\br{1-{p^*}^{r-1}}}$ 
		Further, $\tau_v(1)=\tau_v(0,1,0)$ has distribution $\Po_{<k-2}(cp^*)$ and is independent of $\tau_v(0)$.
	\end{description}
	Since the first sub-case occurs with probability $\bar q$ and the second one accordingly with probability $1-\bar q$, we obtain
	the generating function $f_{001}$.
\item[Case 3: $(\mu_v^*,\mu_{v\toboss}^*,\mu_{\fromboss v}^*)=010$:] 
	Because $\mu_v^*=0$ and $\mu_{\toboss v}^*=1$, we have $\tau_v(1)=k-1$ with certainty.
	Further, because $\mu_{\fromboss v}^*=0$ and $\tau_v(1)=k-1$, (\ref{eqKathrin1e}) entails that $\mu_{\fromboss b}^*=1-\mu_{b\toboss}$
		for all children $b$ of $v$.
	Hence, $\tau(0)=\tau(0,0,1)=\Po\br{c\br{1-{p^*}^{r-1}}}$ and $\tau_v(1)=\tau_v(0,1,0)=k-1$. 
	Thus, the offspring distribution of $v$ is given by $g_{010}$.
\item[Case 4: $(\mu_v^*,\mu_{ v\toboss}^*,\mu_{\fromboss v}^*)=110$:] 
	Since $\mu_v^*=1$, (\ref{eqKathrin1e}) entails that $\mu_{\fromboss b}^*=1$ for all children $b$ of $v$.
	Hence $\tau_v(0)=\tau_v(0,0,1)$ and  $\tau_v(1)=\tau_v(1,1,1)\geq k$.
	Consequently, $\tau_v(1)$ has distribution $\Po_{\geq k}\br{c{p^*}^{r-1}}$ independently of $\tau_v(0)$ with distribution $\Po\br{c\br{1-{p^*}^{r-1}}}$.
	Thus, we obtain $g_{110}$.
\item[Case 5: $(\mu_v^*,\mu_{ v\toboss}^*,\mu_{\fromboss v}^*)=111$:] 
	As in the previous case, $\mu_v^*=1$, (\ref{eqKathrin1e}) ensures that $\mu_{\fromboss b}^*=1$ for all children $b$ of $v$.
	Thus, $\tau_v(0)=\tau_v(0,0,1)$.
	Furthermore, as $\mu_v^*=\mu_{\fromboss v}^*=1$, $\tau_v(1)=\tau_v(1,1,1)$ has distribution $\Po_{\geq k-1}(c{p^*}^{r-1})$.
	In summary, the distribution of the offspring of $v$ is given by $g_{111}$.
	\end{description}
Thus, in every possible case the distribution of $\tau_x$ coincides with the offspring distributions determined by the generating functions in Figures~\ref{Fig_g} and~\ref{Fig_f}.
\qquad
\end{proof}

{\em Proof of \Thm~\ref{Thm_lwc}.}\\
By \Prop~\ref{Prop_main} we have
	$$\lim_{n\to\infty}\Lambda_{d,r,k,n}=\atom_{\theta_{d,r,k}}\quad \text{and}\quad \theta_{d,r,k}=\lim_{t\to\infty}\cL([\T_t(d,r,k)]).$$
Moreover, combining \Lem s~\ref{Lemma_sizeless}, \ref{Lemma_topdown}, \ref{Lemma_decorated} and~\ref{Lem_id}, we see that
$\theta_{d,r,k}=\thet_{d,r,k,p^*}$.
\qquad
\endproof

{\em Proof of \Thm~\ref{Thm_main}.}\\
Let  $s\geq0$ and let $\tau$ be a $\cbc{0,1}$-marked rooted acyclic factor graph.
The function
	$f=\kappa_{\tau,s}$
is continuous, and we let
	$$z=\Erw_{\T(d,r,k,p^*)}[f([\T(d,r,k,p^*)])]=\pr\brk{\partial^s[\T(d,r,k,p^*)]=\partial^s[\tau]}.$$
Since $\cP(\fG_{\cbc{0,1}})$ is equipped with the weak topology, the function
	$$F:\cP(\fG_{\cbc{0,1}})\to\RR,\qquad\xi\mapsto\abs{\int f\dd\nu-z},$$
is continuous as well.
Furthermore, \Thm~\ref{Thm_lwc} implies that $\Lambda_{d,r,k,n}$ converges to $\atom_{\thet_{d,r,k,p^*}}$ weakly in $\cP^2(\fG_{\cbc{0,1}})$, i.e.
	\begin{align}\label{eqConvProb1}
	\lim_{n\to\infty}\int F\dd\Lambda_{d,r,k,n}
		=\abs{\int f\dd \thet_{d,r,k,p^*}-z}=|\Erw_{\T(d,k,p^*)}[f([\T(d,r,k,p^*)])]-z|=0.
	\end{align}
Let
	$X_\tau(\vF)=n^{-1}\abs{\cbc{v\in[n]:\partial^s[\vF_v,v,\sigma_{k,\vF_v}]=\partial^s[\tau]}}.$
Plugging in the definition of $\Lambda_{d,r,k,n}$, we obtain
	\begin{align*}
	\int F\dd\Lambda_{d,r,k,n}&=
		\Erw_{\vF}\abs{X_\tau(\vF)-z}.
	\end{align*}
Using (\ref{eqConvProb1}), we obtain that $X_\tau(\vF)$ converges to $z$ in probability.
\qquad
\endproof

\section{Appendix}
\subsection{Proof of \Thm~\ref{Thm_GW}}
\begin{proof}
Let $T$ be an acyclic factor graph and
let
$$
 X_s=X_s(T,w_0)=\frac1n \sum_{v\in [n]}\vecone\cbc{\partial^{s}(\vF,v)\ism\partial^{s}(T,w_0)}
$$
and
$$
 p_s=p_s(T,w_0)=\pr\brk{\partial^{s}(\T(d,r),v_0)\ism\partial^{s}(T,w_0)}.
$$ 
We claim that
\begin{equation}\label{eq_prob}
 X_s\to p_s \quad\mbox{in probability as $n$ tends to infinity}
\end{equation}
for all $s\geq 0$.

As a first step we determine the 
probability that a variable node 
of $\vec F$ lies in a cycle that comprises
$l$ variable node in total.
This probability is bounded 
by 
\begin{equation}
\label{app:boundcyc}
l!\bink{n-1}{l-1}\brk{\bink{n}{r-2}
\frac{d}{n^{r-1}}}^l=O\br{\frac 1n}.
\end{equation}

We proceed with proving  assertion (\ref{eq_prob}) for even distances from the root. Therefore, we let $s=2t$. 
The proof proceeds by induction on $t$. For $t=0$ there is nothing to show, because $\partial^{2t}(\vF,v)$ only consists of $v$. To proceed from $t$ to $t+1$, let $m$ denote the number 
of children $a_1,\ldots,a_m$ of $w_0$ in $T$. For each $1\leq i \leq m$ and $1\leq l \leq r-1$ let $T_i^{(l)}$ denote the tree rooted at the $l$-th child 
$v_i^{(l)}$ of $a_i$ in $T$ in the forest obtained
by deleting $w_0$ from $T$. Finally, for some $\tilde r_i\leq r-1$ let $C_i^{1},\ldots, C_i^{(\tilde r_i)}$ denote the distinct isomorphism classes among 
$\{\partial^{2t}(\T_i^{(l)},v_i^{(l)}) : l\leq r-1\}$ and let $s_i^{(j)}=|\{l\leq r-1 : \partial^{2t}(\T_i^{(l)},v_i^{(l)})\in C_i^{j}\}|$. Let $v$ be an arbitrary variable node in $\vF$.
Our aim is to determine the probability of the event that $\partial^{2(t+1)}(\vF,v)$ is isomorphic to $\partial^{2t}(T,w_0)$. For that reason we think of $\vH$ being created in two
rounds. First we insert every possible edge in $[n]\setminus {v}$ with probability $p$ independently. Then we complete subset of cardinality $r-1$ of $[n]\setminus {v}$ to an edge with $v$ 
with probability $p$ independently. For the above event to happen $v$ must be contained in exactly $m$ edges. The number of edges that contain $v$ is binomial distributed with 
parameters $\binom{n}{r-1}$ and $p$. Since $\binom{n}{r-1}p\rightarrow \frac{d}{(r-1)!}$ as $n$ tends to infinity, we can approximate this number with a Poisson distribution, and $v$ is contained
in the right number of edges with probability
  \begin{equation*}
   \frac{c^m}{m!\exp(c)} + \mbox{o}(1).
  \end{equation*}
By \eqref{app:boundcyc},
 \whp is not contained in a cycle of
 length $4t+4$, i.e.
 all the pairwise distances between factor nodes connected to $v$ in the
factor graph are at least $4t+2$.
Thus, conditioned on $v$ having the right degree, by induction $v$ forms an edge with exactly $s_j^{(l)}$ vertices with neighbourhood isomorphic to  
$\partial^{2t}(\T_i^{(l)},v_i^{(l)})\in C_i^{j}$ with probability
$$
\binom{r-1}{s_j^{(1)},\ldots,s_j^{(\tilde r_j)}}\prod_{i=1}^{\tilde r_j}p_t(C_j^{(i)}) + \mbox{o}(1)
$$
independently for all $c$ edges $a_j$. Therefore
$$
    \Erw_{\vF}[X_{2(t+1)}]=\frac{\br{c}^m}{m!\exp\br{c}}
    \prod_{j=1}^m\br{
    \binom{r-1}{s_j^{(1)},\ldots,s_j^{(\tilde r_j)}}\prod_{i=1}^{\tilde r_j}p_{2t}(C_j^{(i)})} + \mbox{o}(1).
$$
By definition of $\T(d,r)$, we obtain $\Erw_{\vF}[X_{2(t+1)}]=p_{t+1} + \mbox{o}(1)$. As a next step we determine $\Erw_{\vF}[X_{2(t+1)}^2]$ to apply Chebyshev's inequality. 
Let $\vec v,\vec w\in[n]$ be two randomly chosen vertices. Then \whp\ 
$\vec v$ and $\vec w$ have distance 
at least $4t+4$ in the factor graph $\vF$, conditioned on which $\partial^{2(t+1)}(\vF,\vec v)$ and $\partial^{2(t+1)}(\vF,\vec w)$ are independent. Therefore, we obtain
\begin{align*}
   &\pr_{\vec v,\vec w}\br{\partial^{2(t+1)}(\vF,\vec v)\ism\partial^{2(t+1)}(T,w_0)\wedge\partial^{2(t+1)}(\vF,\vec w)\ism\partial^{2(t+1)}(T,w_0)}\\
   &=\pr_{\vec v}\br{\partial^{2(t+1)}(\vF,\vec v)\ism\partial^{2(t+1)}(T,w_0)}
    \pr_{\vec w}\br{\partial^{2(t+1)}(\vF,\vec w)\ism\partial^{2(t+1)}(T,w_0)}+\mbox{o}(1)\\
  \end{align*}
And finally
  \begin{align*}
    &\Erw_{\vF}[X_{2(t+1)}^2]&\\
    &\;\;=\Erw_{\vF}\brk{\pr_{\vec v}\br{\partial^{2(t+1)}(\vF,\vec v)\ism\partial^{2(t+1)}(T,w_0)}
    \pr_{\vec w}\br{\partial^{2(t+1)}(\vF,\vec w)\ism\partial^{2(t+1)}(T,w_0)}}\\
    &\;\;\;+\frac 1n \Erw_{\vF}[X_{2(t+1)}] +\mbox{o}(1)\\
    &\;\;=\Erw_{\vF}[X_{2(t+1)}]^2 +\mbox{o}(1).
  \end{align*}
Assertion (\ref{eq_prob}) follows from Chebyshev's inequality. Since every factor node in $\vF$ and $\T(d,r)$ has degree exactly $r$ the assertion follows for odd $t$. 

To prove the statement of the Theorem it suffices to show that $\int f\dd\Lambda_{d,r,n}$ converges to $f(\cL([\T(d,r)]))$ for every continuous function $f:\cP(\fG)\to\RR$ with bounded support.
Since $\cP(\fG)$ together with the weak topology is a Polish space, it is metrizable and weak convergence is equivalent to convergence in some metric on 
$\cP(\fG)$. Assertion (\ref{eq_prob}) implies that the distance of $\lambda_{\vF}$ and $\cL([\T(d,r)])$ in this metric converges to $0$ in probability (as $n$ tends to infinity). 
Since $f$ is uniformly continuous this implies that $f(\lambda_{\vF})$ converges to $f(\cL([\T(d,r)]))$ in probability. Since $f(\lambda_{\vF})$ is a bounded random variable 
and $f(\cL([\T(d,r)]))\in\RR$, convergence in probability implies $L_1$ convergence and we obtain
\begin{equation*}
 \abs{\Erw_{\vF}\brk{f(\lambda_{\vF})]-f(\cL([\T(d,r)]))}}\to 0.
\end{equation*}
Plugging in the definition of $\Lambda_{d,r,n}$ we obtain $\int f\dd\Lambda_{d,r,n}=\Erw_{\vF}\brk{f(\lambda_{\vF})]}$, which completes the proof.
\qquad
\end{proof}

\subsection{Tables of Definitions}

We begin with listing the random acyclic factor graphs defined throughout the paper.
$\T(d,r)$ is the unmarked recursive acyclic factor graph in which each variable node has $\Po(d/(r-1)!)$ children independently
and each factor node has $r-1$ children deterministically. From $\T(d,r)$ we may construct labels on variable and factor nodes
\emph{bottom-up} using a variant of Warning Propagation (WP). Note that in all trees defined below it holds that variable nodes are of types $\vertexl$ and factor nodes 
are of types $\edgel.$

\begin{table}[H]
\footnotesize
\caption{Bottom-up Trees}
 \begin{tabular}{p{0.15\textwidth} p{0.26\textwidth} p{0.49\textwidth} lll}
Tree name & Types & Further description\\
$\T_t(d,r,k)$ & $\{0,1\}$ & Obtained from $\T(d,r)$ after $t$ rounds of WP.  Every node is labelled with its WP mark.\\
$\hat{\T}_t(d,r,k)$ & $\labelset$ & Obtained from $\T(d,r)$ after $t$ rounds of 5-type WP.
                                  Every node is labelled with the triple of its WP mark, the message to its parent and the message from its parent.\\
  \end{tabular}
\end{table}

Alternatively, we may construct labels \emph{top-down}, so the labels are constructed simultaneously with the tree.

\begin{table}[H]
\caption{Top-down Trees}
\footnotesize
 \begin{tabular}{p{0.15\textwidth} p{0.26\textwidth} p{0.49\textwidth} lll}
  Tree name & Types & Further description\\
$\hat{\T}(d,r,k,p)$ & $\labelset$ & Constructed according to the generating functions of Figures~\ref{Fig_g} and~\ref{Fig_f}.\\
$\T(d,r,k,p)$ & $\{0,1\}$ & $2$-type projection of $\hat{\T}(d,r,k,p)$ to the first bit.\\
 \end{tabular}
\end{table}

There are two notational conventions. With $\hat{\T}$ we denote $9$-type trees, while $\T_t$ indicates a tree whose labels
were created bottom-up. 
Finally, we have two more trees which allow us in a sense to transition between the top-down and the bottom-up trees.

\begin{table}[H]
\caption{Transition Trees}
\footnotesize
 \begin{tabular}{p{0.15\textwidth} p{0.26\textwidth} p{0.49\textwidth} lll}
  Tree name & Types & Further description\\
$\T^*(d,r,k)$ & $\{0,1\}$ & Labels created top-down, mimic the upwards messages $\mu_{x\toboss}$.\\
$\hat{\T}^*(d,r,k)$ & $\labelset$ & Obtained from $\T^*(d,r,k)$ according to the rules~\eqref{eqKathrin1e}-\eqref{eqKathrin2v}.
 \end{tabular}
\end{table}

Lemma~\ref{Lem_id} says that $\hat{\T}^*(d,r,k)$ has the same distribution as $\hat{\T}(d,r,k,p^*)$.

We define various probability distributions and their corresponding laws in the paper.
We include some equivalences which are not part of the definitions, but which we prove during the paper. As before, we use the shorthand $F_{\vH}=\vF$.

\begin{table}[H]
\caption{Distributions in $\cP(\fF)$, resp. $\cP(\fF_{\{0,1\}})$}
\footnotesize
 \begin{tabular}{p{0.15\textwidth} p{0.33\textwidth} p{0.42\textwidth} lll}
Distribution & Definition & Description\\
$\lambda_{F}$ & $\frac1{|\cV(F)|}\sum_{v\in \cV(F)}\delta_{[F_v,v]}$ & distribution of neighbourhoods of variable nodes in $F$.\\
$\lambda_{k,F}$ & $\frac1{|\cV(F)|}\sum_{v\in\cV(F)}\delta_{[F_v,v,\sigma_{k,F_v}]}$ & nodes labelled according to membership of the core.\\
$\lambda_{k,F,t}$ & $\frac1{|\cV(F)|}\sum_{v\in \cV(F)}\delta_{[F_v,v,\mu_{\nix}(t|F_v)]}$ & vertices labelled after $t$ rounds of WP.\\
\end{tabular}
\end{table}

\begin{table}[H]
\caption{Distributions in $\cP^2(\fF)$, resp. $\cP^2(\fF_{\{0,1\}})$}
\footnotesize
 \begin{tabular}{p{0.15\textwidth} p{0.22\textwidth} p{0.53\textwidth} lll}
Distribution & Definition & Description\\
$\Lambda_{d,r,n}$ & $\Erw_{\vF}[\delta_{\lambda_{\vF}}]$ & distribution of neighbourhoods of the random factor graph \eqref{eqLambdadn}.\\
$\Lambda_{d,r,k,n}$ & $\Erw_{\vF}[\delta_{\lambda_{k,\vF}}]$ & vertices labelled according to membership of the core \eqref{eqLambdadkn}.\\
$\Lambda_{d,r,k,n,t}$ & $\Erw_{\vF}[\lambda_{k,\vF,t}]$ & vertices labelled after $t$ rounds of WP \eqref{eqLambdadknt}.\\
$\Lambda_{d,r,k}$ & $\lim_{n\to \infty} \Lambda_{d,r,k,n} $ & limiting labelled neighbourhood distribution of the random hypergraph.\\
\end{tabular}
\end{table}

\begin{table}[H]
\footnotesize
\caption{Distribution laws}
 \begin{tabular}{p{0.15\textwidth} p{0.35\textwidth} p{0.40\textwidth} lll}
Distribution Law & Definition & Remarks\\
$\thet_{d,r,k,p}$ & $\cL[\T(d,r,k,p)]$ & \\
$\theta_{d,r,k}$ & $\lim_{t\to \infty}\cL[\T_t(d,r,k)]$ & = $\thet_{d,r,k,p^*}$ \; (Proposition~\ref{Prop_main}, \Thm~\ref{Thm_lwc}).\\
$\theta_{d,r,k,t}^{s}$ & $\cL(\partial^s[\T(d,r),v_0,\mu_{\nix\toboss}(t|\T(d,r))])$ & \\
$\theta_{d,r,k}^{s,*}$ & $\cL(\partial^s[\T(d,r),v_0,\mu_{\nix\toboss}^*(s|\T(d,r),s)])$ & $ = \lim_{t\to \infty}\theta_{d,r,k,t}^{s} 
= \cL(\partial^s [\T^*(d,r,k)])$ \\
\end{tabular}
\end{table}


\begin{thebibliography}{31}

\bibitem{Barriers}
 Achlioptas,~D. and Coja-Oghlan,~A. 
 (2008) 
Algorithmic barriers from 
phase transitions
{\em
Proc.~49th FOCS} 793--802.

\bibitem{MolloyAchlioptas}
Achlioptas,~D. and Molloy,~M.
(2015)
The solution space geometry of random linear equations
{\em \RSA} {\bf 46} 197--231.

\bibitem{Polish} 
Aldous,~D. and Lyons,~R.
(2007)
Processes on unimodular random networks
{\em Electronic Journal of Probability}
{\bf 12} 1454--1508. 


\bibitem{Aldous}
Aldous,~D. and Steele,~J. (2004)
The objective method{\rm :} probabilistic combinatorial optimization and local weak convergence.
In 
{\em Probability on discrete structures}
(H.\ Kesten, editor)  Springer.

\bibitem{Behrisch}
Behrisch,~M., Coja-Oghlan,~A. and Kang, ~M.
(2010)
The order of the giant component of random hypergraphs
{\em \RSA} {\bf 36} 149--184.

\bibitem{Behrisch2}
Behrisch,~M., Coja-Oghlan, A. and Kang,~M.
(2014)
Local limit theorems for the giant component of random
hypergraphs
{\em
Combin. Probab. Comput.} {\bf 23} (2014)
331–-366.



\bibitem{BenjaminiSchramm}
Benjamini,~I. and Schramm,~O.
(2001)
Recurrence of distributional limits of finite planar graphs
{\em Electron. J. Probab.} 
{\bf 6} 1--13.

\bibitem{BB}
\Bollobas,~B. (2001)
Random graphs 2nd ed.
{\em Cambridge Univ. Press}

\bibitem{BollobasRiordan}
\Bollobas,~B. and Riordan,~O. (2016)
Counting Connected Hypergraphs 
via the Probabilistic Method
{\em Comb. Prob. Comp.} {\bf 25} (2016)
21–-75.

\bibitem{BordenaveCaputo}
Bordenave,~C. and Caputo,~P.
(2013)
Large deviations of empirical neighborhood distribution in sparse random graphs
{\em Probability Theory and Related Fields}
{\bf 163} 149--222.

\bibitem{Botelho}
Botelho,~F., Wormald,~N. and Nivio,~N.
(2012)
Cores of random r-partite hypergraphs
{\em Informa. Process. Lett.} 
{\bf 112} 314--319.


\bibitem{Encores}
Cain,~J. and Wormald,~N.
(2006)
Encores on cores
{\em Electron. J. Combin.} 
{\bf 13} R81.


\bibitem{kcore}
Coja-Oghlan,~A., Cooley,~O., Kang,~M.
and Skubch,~K.
(2017)
How does the core sit inside the mantle?
{\em \RSA} {\bf 51} 459--482.

\bibitem{Cooper}
Cooper,~C.
(2004)
The cores of random hypergraphs with a given degree sequence
{\em \RSA} {\bf 25} 353--375.

\bibitem{Darling}
Darling,~R. and Norris,~J.
(2008)
Differential equation approximations for Markov chains
{\em Probab. Surv.}
{\bf 5} 37--79.

\bibitem{Dembo}
Dembo,~A. and Montanari,~A.
(2008)
Finite size scaling for the core of large random hypergraphs
{\em Ann. Appl. Probab.}
{\bf 18} 1993--2040.

\bibitem{ER}
\Erdos,~P. and \Renyi, A.
(1960)
On the evolution of random graphs
{\em Magayar Tud. Akad. Mat. Kutato Int. Kozl.} {\bf 5} 17--61.

\bibitem{Ibrahimi}
Ibrahimi,~M., Kanoria,~Y., Kraning,~M. and Montanari,~A.
(2015)
The set of solutions of random XORSAT formulae
{\em Annals of Applied Probability}
{\bf 25} 2743--2808.

\bibitem{JansonLuczak}
Janson,~S. and Luczak,~M.
(2007)
A simple solution to the k-core problem
{\em \RSA} {\bf 30} 50--62.

\bibitem{JansonLuczak2}
Janson,~S. and Luczak,~M.
(2008)
Asymptotic normality of the k-core in random graphs
{\em Ann. Appl. Probab.} {\bf 18}
1085-1137.

\bibitem{JLR}
Janson,~S., {\L}uczak,~T.
and Rucinski,~A.
(2000)
Random Graphs
{\em Wiley}.

\bibitem{Karp}
Karp,~R.
(1990)
The transitive closure of a random digraph
{\em \RSA} {\bf 1} 3--93.

\bibitem{Kim}
Kim,~J.~H. 
(2006)
Poisson cloning model for random graphs
{\em Proceedings of the International Congress of Mathematicians} 873--897.

\bibitem{Mossel}
Maneva,~E., Mossel,~E. and Wainwright,~M.
(2007)
A new look at survey propagation and its generalizations
{\em J. ACM} {\bf 54}.

\bibitem{MM}
M\'ezard,~M. and Montanari,~A.
(2009)
Information, physics and computation
{\em Oxford Univ. Press}.

\bibitem{MolloyCores}
Molloy,~M.
(2005)
Cores in random hypergraphs and Boolean formulas
{\em\RSA} {\bf 27} 124--135.

\bibitem{Molloy}
Molloy,~M.
(2012)
The freezing threshold for $k$-colourings of a random graph
{\em Proc. 43rd STOC} 921--930.

\bibitem{MolloyRestrepo}
Molloy,~M. and Restrepo,~R.
(2013)
Frozen variables in random boolean constraint satisfaction problems
{\em Proc. 24th ACM-SIAM Symposium on Discrete Algorithms} 1306--1318.

\bibitem{Pittel}
Pittel,~B., Spencer,~J. and Wormald,~N.
(1996)
Sudden emergence of a giant $k$-core in a random graph
{\em \JCTB} {\bf 67} 111--151.

\bibitem{Riordan}
Riordan,~O.
(2008)
The $k$-core and branching processes
{\em Combin. Probab. Comput.}
{\bf 17} 111--136.

\bibitem{Pruzan}
Schmidt-Pruzan,~J. and Shamir,~E.
(1985)
Component structure in the evolution of random hypergraphs
{\em Combinatorica} {\bf 5} (1985) 81--94.
\end{thebibliography}
\end{document}